 \documentclass[12pt,a4]{amsart}

\usepackage{amsmath}

 \usepackage{mathptmx} 
\usepackage[utf8]{inputenc}
\usepackage[T1]{fontenc}
\usepackage{ mathrsfs }
\usepackage{ graphicx, longtable, float, wrapfig, soul,
  textcomp, marvosym, wasysym, latexsym, hyperref, bbm,
  fancybox,fancyvrb, cite,amssymb,amsrefs,color,subfig,pdfpages,enumerate}
\usepackage[usenames,dvipsnames]{xcolor}
\usepackage{enumitem}
\usepackage{listings}
\usepackage{tabularx}
\usepackage{mathtools}
 \usepackage{kbordermatrix}
  \usepackage{relsize}
\usepackage[all]{xy}
\newtheorem{theorem}{Theorem}[section]
\newtheorem{lemma}[theorem]{Lemma}
\newtheorem{proposition}[theorem]{Proposition}
\newtheorem{corollary}[theorem]{Corollary}

\theoremstyle{definition}
\newtheorem{example}[theorem]{Example}
\newtheorem{definition}[theorem]{Definition}
 
\newtheorem{open}[theorem]{Open Problem} 
\newtheorem{notation}[theorem]{Notation}

\newcommand{\iso}{Iso}
\newcommand{\aut}{Aut} 

\numberwithin{equation}{section}
\DeclareMathAlphabet{\mcal}{OMS}{cmsy}{m}{n}
\SetMathAlphabet{\mcal}{bold}{OMS}{cmsy}{b}{n}

\begin{document}

\title{$\aleph_0$-categoricity of semigroups II
}
\thanks{This work was funded by EPSRC as part of a PhD at the University of York, supervised by Prof. Victoria Gould.  The author has also received funding from  the Deutsche  Forschungsgemeinschaft (DFG, project number 622397)
and from the European Research Council (Grant Agreement no. 681988, CSP-Infinity).}
\author{Thomas Quinn-Gregson  
}

%


\maketitle

\begin{abstract}
  \noindent A countable semigroup is $\aleph_0$-categorical if it can be characterised, up to isomorphism,  by its first-order properties. In this paper we continue our investigation into the $\aleph_0$-categoricity of semigroups. Our main results are a complete classification of  $\aleph_0$-categorical orthodox completely 0-simple semigroups, and  descriptions of the $\aleph_0$-categorical members of certain classes of strong semilattices of semigroups. 
\end{abstract}

\section{Introduction} 
 
  A countable structure is \textit{$\aleph_0$-categorical} if it is uniquely determined by its first-order properties, up to isomorphism.  While the concept of $\aleph_0$-categoricity arises naturally from model theory,  it has a purely algebraic formulation thanks to the Ryll-Nardzewski Theorem (RNT).  Independently accredited to Engeler \cite{Engeler}, Ryll-Nardzewski \cite{Ryll} and Svenonius \cite{Svenonius}, it states that the $\aleph_0$-categoricity of a structure $M$ is equivalent to there being only finitely many orbits in the natural action of Aut($M$) (the automorphism group of $M$) on $M^n$, for each $n\geq 1$. 
   Significant results exist for both relational  and algebraic structures from the point of view
  of $\aleph_0$-categoricity, but, until recently, little was known in the context of semigroups. This article is the second of a pair  initiating and developing the 
  study of $\aleph_0$-categorical semigroups. For background and motivation we refer the reader to \cite{Hodges93} and \cite{Evans}, and to our first article \cite{Quinncat}.  
  
  We explore in \cite{Quinncat} the behaviour of $\aleph_0$-categoricity with respect to standard constructions, such as quotients and subsemigroups. For example, $\aleph_0$-categoricity of a semigroup is inherited by both its maximal subgroups and its principal factors. Differences with the known theory for groups and rings emerged, for example, any $\aleph_0$-categorical nil ring is nilpotent, but the same is not true for semigroups. 
  While keeping the machinery at a low level, we were able to give, amongst other results,  complete classifications of $\aleph_0$-categorical primitive inverse semigroups and of $E$-unitary inverse semigroups with finite semilattices of idempotents. 

For the work in this current article, it is helpful to develop some general strategies and then apply them in various contexts. In view of this, in Section \ref{basic cat}, we introduce $\aleph_0$-categoricity in the setting of (first-order) structures. Although we will mostly be working in the context of semigroups, this broader view will be useful for studying  structures, such as graphs and semilattices, which  naturally arise in our considerations of semigroups.  Key results from \cite{Quinncat} are given in this  setting.
In particular, we formalise the previously defined concept of \textit{$\aleph_0$-categoricity over a set of subsets}; the $\aleph_0$-categoricity of rectangular bands over any set of subrectangular bands acts as a useful example. 

 In Section \ref{new method}  we construct a handy method for dealing with the $\aleph_0$-categoricity of semigroups in which their automorphisms can be built from certain ingredients. This is then used in Section \ref{sec strong} to study the $\aleph_0$-categoricity of strong semilattices of semigroups. The main results of this article are in Section \ref{Sec rees}, where we continue from \cite{Quinncat} our study into the $\aleph_0$-categoricity of completely 0-simple semigroups. We follow a method of Graham and Houghton by considering graphs arising from Rees matrix semigroups, which necessitated our  study of  $\aleph_0$-categoricity in the general setting of structures. 

We   assume that all structures considered will be of countable cardinality.

\section{The $\aleph_0$-categoricity of a structure} 
\label{basic cat}

We begin by translating a number of results in \cite{Quinncat} to the  general setting of (first-order) structures. Their proofs easily generalize, and as such we shall omit
them, referencing only the corresonding result in \cite{Quinncat}. 

 A (first-order) structure is a set $M$ together with a collection of constants $\mathfrak{C}$, finitary relations $\mathfrak{R}$, and finitary functions $\mathfrak{F}$ defined on $M$. 
  We denote the structure as $(M;\mathfrak{R},\mathfrak{F},\mathfrak{C})$, or simply $M$ where no confusion may arise. Each constant element is associated with a constant symbol, each $n$-ary relation is associated with an $n$-ary relational symbol, and each $n$-ary function is associated with an $n$-ary function symbol. The collection $L$ of these symbols is called the \textit{signature of $M$}. We follow the usual convention of not distinguishing between the constants/relations/functions of $M$, and their
  corresponding abstract symbols in $L$.

Our main example is that of  a semigroup $(S,\cdot)$, where $S$  is a set  together with a single (associative) binary operation $\cdot\,$, and so the associated signature  consists of a single binary function symbol. 
  
A property of a structure is  \textit{first-order} if it can be formulated within first-order predicate calculus.  A (countable) structure is \textit{$\aleph_0$-categorical} if it can be uniquely classified by its first-order properties, up to isomorphism.

 The central result in the study of $\aleph_0$-categorical structures is the  Ryll-Nardzewski Theorem, which translates the concept to the study of oligomorphic automorphism groups (see \cite{Hodges93}). Before stating it, it is worth fixing some notation and definitions. Let $\phi\colon A\rightarrow B$ be a map,  let $\overline{a}=(a_1,\dots,a_n)$  be an $n$-tuple  of $A$ and let $M\subseteq A$. Then we let $\overline{a}\phi$ denote the $n$-tuple of $B$ given by $(a_1\phi,\dots,a_n\phi)$, and $M\phi$ denotes the subset $\{m\phi: m\in M\}$ of $B$. 

Given a structure $M$, we say that a pair of $n$-tuples $\overline{a}=(a_1, \dots, a_n)$ and $\overline{b}=(b_1, \dots ,b_n)$ of $M$ are  \textit{automorphically equivalent} or
{\em belong to the same $n$-automorphism type} if there exists an automorphism $\phi$ of $M$ such that  $\overline{a}\phi=\overline{b}$, that is, $a_i\phi = b_i$ for each $i
\in \{ 1,\hdots, n\}$. We denote this equivalence relation as $\overline{a} \, \sim_{M,n} \, \overline{b}$. We call Aut($M$) \textit{oligomorphic} if Aut($M$) has only finitely many orbits in its action on $M^n$ for each $n\geq 1$, that is, if each $|M^n/\sim_{M,n}|$ is finite. 

\begin{theorem}[The Ryll-Nardzewski Theorem (RNT)] 
A structure $M$ is $\aleph_0$-categorical if and only if Aut($M$) is oligomorphic. 
\end{theorem}  

It follows from the RNT that every $\aleph_0$-categorical structure is \textit{uniformly locally finite} \cite[Corollary 7.3.2]{Hodges93}, that is, there is a finite uniform bound on the size of the $n$-generated substructures, for each $n\geq 1$. In particular, an $\aleph_0$-categorical semigroup is periodic, with bounded index and period. 

Another immediate consequence of the RNT is that any characteristic substructure  inherits $\aleph_0$-categoricity, where a subset/substructure is called \textit{characteristic} if it is invariant under automorphisms of the structure.
 However, key subsemigroups of a semigroup such as maximal subgroups and principal ideals are not necessarily characteristic, and a more general definition is required: 
 
 \begin{definition}\label{fprc}   Let $M$ be a structure and, for some fixed $t\in \mathbb{N}$, let  $\{\overline{X}_i:i\in I\}$ be a collection of $t$-tuples of $M$.
  Let $\{A_i:i\in I\}$ be a collection of subsets of $M$ with the property that for any automorphism $\phi$ of $M$ such that there exists $i,j\in I$ with $\overline{X}_i \phi = \overline{X}_j$, then $\phi|_{A_i}$ is a bijection from $A_i$ onto $A_j$.
   Then we call $\mcal{A}=\{(A_i,\overline{X}_i):i\in I\}$ a \textit{system of $t$-pivoted pairwise relatively characteristic ($t$-pivoted p.r.c.) subsets} (or, substructure,
if each $A_i$ is a substructure) of $M$.
 The $t$-tuple $\overline{X}_i$ is called the \textit{pivot} of $A_i$ ($i\in I$).
   If $|I|=1$ then, letting $A_1=A$ and $\overline{X}_1=\overline{X}$,  we write $\{(A, X)\}$ simply as $(A, X)$, and call $A$ an \textit{$\overline{X}$-pivoted relatively characteristic ($\overline{X}$-pivoted r.c.) subset/substructure} of $M$. 
 \end{definition} 

In \cite{Quinncat}, Definition~\ref{fprc} was shown to be of use in regard to, for example,  Green's relations.
In particular, $\{(H_e,e):e\in E(S)\}$ forms a system of 1-pivoted p.r.c. subgroups of  a semigroup $S$.
 It then followed from the proposition below  that maximal subgroups inherit $\aleph_0$-categoricity, and moreover there exists only finitely many non-isomorphic maximal subgroups in an $\aleph_0$-categorical semigroup.  

\begin{proposition}\cite[Proposition 3.3]{Quinncat} \label{rel-char cat} Let $M$ be an $\aleph_0$-categorical structure and  $\{(A_i,\overline{X}_i):i\in I\}$ a system of $t$-pivoted p.r.c. subsets of $M$.
Then $\{|A_i| : i \in I\}$ is finite.
 If, further, each $A_i$ forms a substructure of $M$, then $\{A_i: i \in I\}$ is finite, up to isomorphism, with each $A_i$ $\aleph_0$-categorical. 
\end{proposition} 

We  use the RNT in conjunction with \cite[Lemma 2.8]{Quinncat} to prove that a structure $M$ is $\aleph_0$-categorical in the following way. For each $n\in \mathbb{N}$, let $\gamma_1,\dots,\gamma_r$ be a finite list of equivalence relations on $M^n$ such that $M^n/\gamma_i$ is finite for each $1\leq i \leq r$ and 
\[ \gamma_1\cap \gamma_2 \cap \cdots \cap \gamma_r \subseteq \, \sim_{M,n}. 
\] 
 A consequence of the two aforementioned results is that $M$ is $\aleph_0$-categorical. 
 This result will often be drawn upon in a less formal way as follows. Suppose that we have an equivalence relation $\sigma$ on $M^n$ that arises from different ways in which a given condition may be fulfilled; if $M^n/\sigma$ is finite, then we say the condition has {\em finitely many choices}.
  
\begin{example} Recalling  \cite[Example 2.10]{Quinncat}, consider the equivalence $\natural_{X,n}$ on $n$-tuples of a set $X$ given  by
\begin{equation} \label{natural n}   (a_1,\dots,a_n) \, \natural_{X,n} \, (b_1,\dots,b_n) \text{ if and only if } [a_i=a_j \Leftrightarrow b_i=b_j, \text{ for each } i,j].
\end{equation} 
A pair of $n$-tuples $\overline{a}$ and $\overline{b}$ are $\natural_{X,n}$-equivalent if and only if there exists a bijection $\phi\colon\{a_1,\dots,a_n\} \rightarrow \{b_1,\dots,b_n\}$ such that $a_i\phi=b_i$, and the number of $\natural_{X,n}$-classes of $X^n$ is  finite, for each $n\in \mathbb{N}$.
 Note also that if $M$ is a structure then any pair of $n$-automorphically equivalent tuples are clearly $\natural_{M,n}$-equivalent. 
\end{example} 

Let $M$ be a structure and $\mcal{A}=\{A_i:i\in I\}$ a collection of subsets of $M$. We may extend the signature of $M$ to include the unary relations $A_i$ ($i\in I$). 
We denote the resulting structure as  $\underline{M}=(M;\mcal{A})$, which we call a \textit{set extension of $M$}.
 If $\mcal{A}=\{A_1,\dots,A_n\}$ is finite, then we may simply write $\underline{M}$ as $(M;A_1,\dots,A_n)$. 

 Notice that automorphisms of $\underline{M}$ are simply those automorphisms of $M$ which fix each $A_i$ setwise, that is automorphisms $\phi$ such that $A_i\phi=A_i$ ($i\in I$).
  The set of all such automorphisms will be denoted Aut($M;\mcal{A}$), and clearly forms a subgroup of Aut($M$).
  The $\aleph_0$-categoricity of $\underline{M}$ is therefore equivalent to our previous notion of $M$ being  \textit{$\aleph_0$-categorical over $\mcal{A}$} in \cite{Quinncat}. 

\begin{lemma} \cite[Lemma 5.2]{Quinncat} Let $M$ be a structure with a system of $t$-pivoted p.r.c. subsets $\{(A_i,\overline{X}_i):i\in I\}$. 
Then $(M;\{A_i:i\in I\})$ is $\aleph_0$-categorical if and only if $M$ is $\aleph_0$-categorical and $I$ is finite.
\end{lemma} 

\begin{lemma}\cite[Lemma 5.3]{Quinncat} \label{over char 2} Let $M$ be a structure, let $t,r\in\mathbb{N}$, and for each $k\in \{1,\dots,r\}$ let $\overline{X}_k\in M^t$. Suppose also that   $A_k$ is an $\overline{X}_k$-pivoted relatively characteristic subset of $M$ for $1\leq k \leq r$. 
  Then $(M;A_1,\dots,A_r)$ is $\aleph_0$-categorical if and only if $M$ is $\aleph_0$-categorical. 
\end{lemma} 

Consequently, if $S$ is an $\aleph_0$-categorical semigroup and $G_1,\dots, G_n$ is a collection of maximal subgroups of $S$ then $(S;G_1,\dots,G_n)$ is $\aleph_0$-categorical. 

However, note that not every $\aleph_0$-categorical set extension of a semigroup requires the subsets to be relatively characteristic. We claim that any set extension of a rectangular band by a finite set of subrectangular bands is $\aleph_0$-categorical. This result is of particular use in the next section when considering the $\aleph_0$-categoricity of  normal bands. 

Recall that every rectangular band can be written as a direct product of a left zero and right zero semigroup.
 The following  isomorphism theorem for rectangular bands will be vital for proving our claim, and follows immediately from \cite[Corollary 4.4.3]{Howie94}: 

\begin{lemma}\label{rb iso} Let $B_1=L_1\times R_1$ and $B_2=L_2\times R_2$ be a pair of rectangular bands. If $\phi_L\colon L_1\rightarrow L_2$ and $\phi_R\colon R_1\times R_2$ are a pair of bijections, then the map $\phi\colon B_1\rightarrow B_2$ given by $(l,r)\phi=(l\phi_L,r\phi_R)$ is an isomorphism, denoted $\phi=\phi_L\times \phi_R$.
 Conversely, every isomorphism can be constructed this way. 
\end{lemma} 

\begin{theorem}\label{RB cat} If $B$ is a rectangular band and $B_1,\dots,B_r$ is a finite list of subrectangular bands of $B$, then $\underline{B}=(B;B_1,\dots,B_r)$ is $\aleph_0$-categorical.
 In particular, a rectangular band is $\aleph_0$-categorical. 
\end{theorem} 

\begin{proof} Let $B=L\times R$, where $L$ is a left zero semigroup and $R$ is a right zero semigroup. For each $1\leq k \leq r$, let $L_k\subseteq L$ and $R_k\subseteq R$ be such that $B_k=L_k\times R_k$.  
 Define a pair of equivalence relations $\sigma_L$ and $\sigma_R$  on $L$ and $R$, respectively, by
\begin{align*}  & i \, \sigma_L \, j \Leftrightarrow [i\in L_k \Leftrightarrow j\in L_k, \text{ for each } k], \\
 & i \, \sigma_R \, j \Leftrightarrow [i\in R_k \Leftrightarrow j\in R_k, \text{ for each } k].
\end{align*}
The equivalence classes of $\sigma_L$ are simply the set $L\setminus \bigcup_{1\leq k \leq r} L_k$ together with certain intersections of the sets  $L_k$. Since $r$ is finite, it follows that $L/\sigma_L$ is finite, and similarly $R/\sigma_R$ is finite. 
Let $\overline{a}=((i_1,j_1),\dots,(i_n,j_n))$ and $\overline{b}=((k_1,\ell_1),\dots,(k_n,\ell_n))$ be a pair of $n$-tuples of $B$ under the four conditions that 
\begin{enumerate}[label=(\arabic*), font=\normalfont]
\item $i_s \, \sigma_L \, k_s$ for each $1\leq s \leq n$, 
\item $j_s \, \sigma_R \, \ell_s$ for each $1\leq s \leq n$,
\item $(i_1,\dots,i_n) \, \natural_{L,n} \, (k_1,\dots,k_n)$, 
\item $(j_1,\dots,j_n) \, \natural_{R,n} \, (\ell_1,\dots,\ell_n)$, 
\end{enumerate}
where $\natural_{L,n}$ and $\natural_{R,n}$ are the equivalence relations given by \eqref{natural n}. 
 By conditions (3) and (4), there exists bijections 
 \[ \phi_L\colon \{i_1,\dots,i_n\}\rightarrow \{k_1,\dots,k_n\} \text{ and } \phi_R\colon \{j_1,\dots,j_n\}\rightarrow \{\ell_1,\dots,\ell_n\}
 \] 
 given by $i_s\phi_L=k_s$ and $j_s\phi_R=\ell_s$ for each $1\leq s\leq n$. 
By condition (1), we can pick a bijection $\Phi_L$ of $L$ which extends $\phi_L$ and fixes each $\sigma_L$-classes setwise, and similarly construct $\Phi_R$.
 Then $\Phi=\Phi_L\times \Phi_R$ is an automorphism of $B$. Moreover, if $(i,j)\in B_k$ then $i\in L_k$ and as $i \, \sigma_L \, (i\Phi_L)$ we have  $i\Phi_L\in L_k$.
 Dually, $j\in R_k$ and as $j \, \sigma_R \, (j\Phi_R)$ we have $j\Phi_R\in R_k$. 
 Hence there exists $\ell\in L$ and $r\in R$  such that $(i\Phi_L,r)$ and $(\ell,j\Phi_R)$ are in $B_k$, so that 
\[ (i\Phi_L,r) (\ell,j\Phi_R) = (i\Phi_L,j\Phi_R)\in B_k
\] 
as $B_k$ is a subrectangular band.
 We have thus shown that $(i,j)\Phi=(i\Phi_L,j\Phi_R)\in B_k$, and so $B_k\Phi\subseteq B_k$. 
We observe that $\Phi^{-1}=\Phi^{-1}_L\times \Phi^{-1}_R$ is also an automorphism of $B$ with $\Phi^{-1}_L$  and $\Phi^{-1}_R$ setwise fixing the $\sigma_L$-classes and $\sigma_R$-classes, respectively. 
Following our previous argument we have $B_k\Phi^{-1}\subseteq B_k$, and so $B_k\Phi=B_k$ for each $k$. Thus $\Phi$ is an automorphism of $\underline{B}$, and is such that 
\[ (i_s,j_s)\Phi=(i_s\Phi_L,j_s\Phi_R)=(i_s\phi_L,j_s\phi_R)=(k_s,\ell_s) 
\] 
for each $1\leq s \leq n$, so that $\overline{a} \, \sim_{\underline{B},n} \, \overline{b}$. Hence, as each of the four conditions on $\overline{a}$ and $\overline{b}$ have finitely many choices, it follows that $\underline{B}$ is $\aleph_0$-categorical.  
\end{proof}

Note that any set can be considered as a structure with no relations, functions or constants. Every bijection of the set is therefore an automorphism, and as such all sets are easily shown to be $\aleph_0$-categorical.
 In fact a simplification of the proof of Theorem \ref{RB cat} gives:  

\begin{corollary}\label{set plus cat} Let $M$ be a set, and $M_1,\dots,M_r$ be a finite list of subsets of $M$.
 Then $(M;M_1,\dots,M_r)$ is $\aleph_0$-categorical. 
\end{corollary}  

\section{A new method: $(M,M';\underline{N};\Psi)$-systems}  \label{new method} 

For many of the structures we will consider, automorphisms can be built from isomorphisms between their components. For example, for a strong semilattice of semigroups $S=[Y;S_{\alpha};\psi_{\alpha,\beta}]$, we can construct automorphisms of $S$ from certain isomorphisms between the semigroups $S_{\alpha}$.
 In this example we also require an automorphism  of the semilattice $Y$, which acts as an   indexing set for the semigroups $S_{\alpha}$. 
  We now extend this idea by setting up some formal machinery to deal with structures in which the automorphisms are built from a collection of data. 

\begin{notation} Given a pair of structures $M$ and $M'$, we let \iso$(M;M')$ denote the set of all isomorphisms from $M$ onto $M'$.
\end{notation}

\begin{definition} Let $M$ be an $L$-structure with fixed substructure $M'$.
 Let~$\mcal{A} = \{M_i:i\in N\}$~be a set of substructures of $M'$ indexed by some $K$-structure $N$ such that $M'=\bigcup_{i\in N} M_i$. 
 Let $N_1,\dots,N_r$ be a finite partition of $N$, and set $\underline{N}=(N;N_1,\dots,N_r)$. 
For each $i,j\in N$, let $\Psi_{i,j}$ be a subset of $\text{Iso}(M_i;M_j)$ under the conditions that 
\begin{enumerate}[itemsep=1ex, leftmargin=0.9cm]
\item[(3.1)]   if $i,j\in N_k$ for some $1\leq k\leq r$ then $\Psi_{i,j} \neq \emptyset$,
\item[(3.2)] if $\phi\in \Psi_{i,j}$ and $\phi'\in \Psi_{j,\ell}$ then $\phi\phi'\in \Psi_{i,\ell}$,
\item[(3.3)] if $\phi\in \Psi_{i,j}$ then $\phi^{-1}\in \Psi_{j,i}$, 
\item[(3.4)] if $\pi\in \text{Aut}(\underline{N})$ and $\phi_i\in \Psi_{i,i\pi}$ for each $i\in N$, then there exists an automorphism of $M$ extending the $\phi_i$.  
\end{enumerate}
 Letting $\Psi=\bigcup_{i,j\in N} \Psi_{i,j}$, then, under the conditions above, we call $\mcal{A}=\{M_i:i\in N\}$  an $(M,M';\underline{N};\Psi)$-system (in $M$).
  If $M'=M$ then we may simply refer to this as an $(M;\underline{N};\Psi)$-system. 
\end{definition} 
By Condition (3.1) if $i,j\in N_k$ for some $k$, then $M_i\cong M_j$. Hence the number of isomorphism types in $\mcal{A}$ is bounded by $r$. Moreover, it follows from Conditions (3.1), (3.2), and (3.3) that $\Psi_{i,i}$ is a subgroup of $\text{Aut}(M_i)$, for each $i\in N$.
If the sets $M_i$ are not pairwise disjoint, then Condition (3.4) should be met with caution. 
Indeed, if $x\in M_i\cap M_j$ then by taking $\pi$ to be the identity map of $\underline{N}$, we have that $x\phi_i,x\phi_j\in M_i\cap M_j$ for all $\phi_i\in \text{Aut}(M_i)$ and $\phi_j\in \text{Aut}(M_j)$ by Condition (3.4).
 However, for our work the sets $M_i$ will mostly be pairwise disjoint, or will all intersect at an element which is fixed by every isomorphism between the $M_i$.
  For example, $M$ could be a semigroup containing a zero, and 0 is the intersection of each of the sets $M_i$. 

Note also that no link needs to exist between the signatures $L$ and $K$. For most of our examples they will be the signature of semigroups and the signature of sets (the empty signature), respectively. 

  Given an $(M;M';\underline{N};\Psi)$-system  $\mcal{A}=\{M_i:i\in N\}$ in $M$, we aim to show that, if $N$ is $\aleph_0$-categorical and each $M_i$ possess a stronger notion of $\aleph_0$-categoricity, then $M$ is $\aleph_0$-categorical. 
  The stronger notion that we require comes from the following definition, which generalises the notion of  $\aleph_0$-categoricity of set extensions. 

\begin{definition} Let $M$ be a structure and $\Psi$ a subgroup of Aut($M$). Then we say that that $M$ is \textit{$\aleph_0$-categorical over $\Psi$} if $\Psi$ has only finitely many orbits in its action on $M^n$ for each $n\geq 1$.
 We denote  the resulting equivalence relation on $M^n$ as $\sim_{M,\Psi,n}$.
\end{definition} 

By taking $\Psi$ to be those automorphisms which fix certain subsets of $M$ we recover our original definition of  $\aleph_0$-categoricity of a set extension. 
Similarly, by taking $\Psi$ to be those automorphisms which preserve a fixed equivalence relation, or those which fix certain equivalence classes, we obtain a pair of notions defined in \cite{Quinncat}. 

\begin{lemma}\label{parts hard sub} 
Let $M$ be a structure, and  $\mcal{A}=\{M_i:i\in N\}$ be an $(M,M';\underline{N};\Psi)$-system.
 If $\underline{N}$ is $\aleph_0$-categorical and each $M_i$ is $ \aleph_0$-categorical over $\Psi_{i,i}$ then 
\[ |(M')^n/\sim_{M,n}|<\aleph_0
\] 
for each $n\geq 1$. 
\end{lemma} 

\begin{proof} Let $\underline{N}=(N;N_1,\dots,N_r)$ and, for each $1\leq k \leq r$, fix some $m_k\in N_k$. For each $i\in N_k$, let $\theta_i\in \Psi_{i,m_k}$, noting that such an element exists by Condition (3.1) on $\Psi$.  
Let $\overline{a}=(a_1,\dots,a_n)$ and $ \overline{b}=(b_1,\dots,b_n)$ be a pair of $n$-tuples of $M'$, with $a_t\in M_{i_t}$ and $b_t\in M_{j_t}$,  and such that $(i_1,\dots,i_n) \, \sim_{\underline{N},n} \, (j_1,\dots,j_n)$ via $\pi\in \text{Aut}(\underline{N})$, say.
 For each $1 \leq k\leq r$, let $i_{k1},i_{k2},\dots ,i_{kn_k}$ be the entries of $(i_1,\dots,i_n)$ belonging to $N_k$, where $k1<k2<\cdots <kn_k$, and set 
\[ \overline{a}_k=(a_{k1},\dots,a_{kn_k})\in (M')^{n_k}.
\]
 We similarly form each $\overline{b}_k$, observing that as $i_t\pi=j_t$ for each $1\leq t \leq n$ and $\pi$ fixes the sets $N_j$ setwise ($1\leq j \leq r$) the elements $j_{k1},j_{k2},\dots ,j_{kn_k}$ are precisely the entries of $(j_1,\dots,j_n)$ belonging to $N_k$, so that $\overline{b}_k=(b_{k1},\dots,b_{kn_k})$ for some $b_{kt}\in M'$.
  Notice that as $N_1,\dots,N_r$ partition $N$ we have $n=n_1+n_2+\cdots + n_r$. Since $i_{kt},j_{kt}\in N_k$ for each $1 \leq t \leq n_k$, we have that $a_{kt}\theta_{i_{kt}}$ and $b_{kt}\theta_{j_{kt}}$ are elements of $M_{m_k}$.
   We may thus suppose further that for each $1\leq k \leq r$,  
\[ (a_{k1}\theta_{i_{k1}},\dots,a_{{kn_k}}\theta_{i_{kn_k}}) \, \sim_{M_{m_k},\Psi_{m_k,m_k},n_k} \, (b_{k1}\theta_{j_{k1}},\dots,b_{{kn_k}}\theta_{j_{kn_k}})
\] 
 via $\sigma_k\in \Psi_{m_k,m_k}$, say (where if $\overline{a}_k$ is a 0-tuple, then we take  $\sigma_k$ to be the identity of $N_{m_k}$). For each $1\leq k \leq r$ and each $i\in N_k$, let 
 \[ \phi_i=\theta_i \sigma_k \theta_{i\pi}^{-1}\colon M_i\rightarrow M_{i\pi},
 \]
  noting that $\phi_i\in \Psi_{i,i\pi}$  by Conditions (3.2) and (3.3) on $\Psi$, since $\theta_i,\sigma_k$ and $\theta_{i\pi}$ are elements of $\Psi$. 
  Hence, by Condition (3.4) on $\Psi$, there exists an automorphism $\phi$ of $M$ extending each $\phi_i$. For any $1\leq k \leq r$ and any $1 \leq t \leq n_k$ we have  
\[ a_{kt}\phi= a_{kt} \phi_{i_{kt}}= a_{kt} \theta_{i_{kt}} \sigma_k \theta_{i_{kt}\pi}^{-1} = b_{kt}\theta_{j_{kt}} \theta_{j_{kt}}^{-1} = b_{kt}, 
\] 
and so $\overline{a} \, \sim_{M,n} \, \overline{b}$ via $\phi$.
 Since $\underline{N}$ is $\aleph_0$-categorical and each $M_i$ are $\aleph_0$-categorical over $\Psi_{i,i}$, the conditions imposed on the tuples $\overline{a}$ and $\overline{b}$ have finitely many choices, and so $|(M')^n/\sim_{M,n}|$ is finite.  
\end{proof} 

By Corollary \ref{set plus cat}, the structure $N$ in the lemma above can simply be a set.
In most cases we take $M'=M$, and the result simplifies accordingly by the RNT as follows. 

\begin{corollary} \label{parts hard} Let $M$ be a structure, and  $\mcal{A}=\{M_i:i\in N\}$ be an $(M;\underline{N};\Psi)$-system. 
If $\underline{N}$ is $\aleph_0$-categorical and each $M_{i}$ is $ \aleph_0$-categorical over $\Psi_{i,i}$, then $M$ is $\aleph_0$-categorical. 
\end{corollary} 

\begin{example} The corollary above  could be used to efficiently prove     the interplay of $\aleph_0$-categoricity  and the greatest 0-direct decomposition of a semigroup with zero \cite[Theorem 4.8]{Quinncat}. 
Indeed, if $S=\bigsqcup^0_{i\in I} S_i$ is the greatest 0-direct decomposition of $S$,  and $I_1,\dots,I_n$ is a finite partition of $I$ corresponding to the isomorphism types of the summands of $S$, then it is a simple exercise to show that $\mcal{S}=\{S_i:i\in I\}$ is an $(S;(I;I_1,\dots,I_n);\Psi)$-system, where $\Psi$ is the collection of all isomorphisms between summands.
 Since $(I;I_1,\dots,I_n)$ is $\aleph_0$-categorical, it follows by Corollary \ref{parts hard} that $S$ is $\aleph_0$-categorical if each $S_i$ is $\aleph_0$-categorical (over $\Psi_{i,i}=\text{Aut}(S_i)$). 
\end{example} 

\section{Strong semilattices of semigroups} \label{sec strong} 

In this section we study the $\aleph_0$-categoricity of strong semilattices of semigroups  by making use of our most recent methodology.
 We are motivated by the work of the author in \cite{Quinnband} and \cite{Quinninv}, where the homogeneity of bands and inverse semigroups are shown to depend heavily on the homogeneity of strong semilattices of rectangular bands and groups, respectively.
 Recall that a structure is \textit{homogeneous} if every isomorphism between finitely generated substructures extend to an automorphism.
  A uniformly locally finite homogeneous structure is $\aleph_0$-categorical \cite[Corollary 3.1.3]{Pherson}.
 Consequently, each homogeneous band is $\aleph_0$-categorical, although the same is not true for homogeneous inverse semigroups.  

While there has not yet been a general study into $\aleph_0$-categorical semilattices, a complete classification of countable homogeneous semilattices was completed in \cite{Droste} and \cite{DrosteTruss}. Since semilattices are uniformly locally finite, this provids us with  a countably infinite collection  of $\aleph_0$-categorical semilattices. For example, the linear order $\mathbb{Q}$ is a homogeneous semilattice, and all $\aleph_0$-categorical linear orders are classified in \cite{Ros69}.

  Let $Y$ be a semilattice. To each $\alpha \in Y$ associate a semigroup $S_{\alpha}$, and assume that $S_{\alpha} \cap S_{\beta} = \emptyset$ if $\alpha \neq \beta$.
   For each pair $\alpha, \beta \in Y$ with $\alpha \geq \beta$, let   $\psi_{\alpha, \beta}\colon  S_{\alpha} \rightarrow S_{\beta}$ be a morphism such that  $\psi_{\alpha, \alpha}$ is  the   identity   mapping  and if  $\alpha \geq \beta \geq \gamma$ then $  \psi_{\alpha, \beta} \psi_{\beta, \gamma} = \psi_{\alpha, \gamma}$. 
On the set $S=\bigcup_{\alpha \in Y} S_{\alpha}$ define a multiplication by 
\[ a * b = (a \psi_{\alpha, \alpha \beta})(b \psi_{\beta, \alpha \beta})
\] 
for $a\in S_{\alpha}, b \in S_{\beta}$, and denote the resulting structure by  $S=[Y;S_{\alpha}; \psi_{\alpha, \beta}]$. Then $S$ is a semigroup, and is called a \textit{strong semilattice $Y$ of the semigroups $S_{\alpha}$} ($\alpha\in Y$). The semigroups $S_{\alpha}$ are called the \textit{components} of $S$. We follow the convention of denoting an element $a$ of $S_{\alpha}$ as $a_{\alpha}$.

 The idempotents of $S=[Y;S_{\alpha}; \psi_{\alpha, \beta}]$ are given by  $E(S)=\bigcup_{\alpha\in Y} E(S_{\alpha})$, and if $E(S)$ forms a subsemigroup of $S$ then 
 \[ E(S) = [Y;E(S_{\alpha});\psi_{\alpha,\beta}|_{E(S_{\alpha})}]. 
\]  
We build automorphisms of strong semilattices of semigroups in a natural way using the following well known result. A proof can be found in \cite{Quinn}. 

\begin{theorem}\label{iso strong}  Let $S=[Y; S_{\alpha}; \psi_{\alpha, \beta}]$ be a  strong semilattices of semigroups. Let $\pi\in \text{Aut}(Y)$ and, for each $\alpha \in Y$, let $\theta_{\alpha}\colon  S_{\alpha}\rightarrow S_{\alpha \pi}$ be an isomorphism. Assume further that for any $\alpha \geq \beta$, the diagram  
\begin{align} \label{1} \xymatrix{
S_{\alpha} \ar[d]^{\psi_{\alpha, \beta}} \ar[r]^{\theta_{\alpha}} &S_{\alpha \pi} \ar[d]^{\psi_{\alpha \pi, \beta \pi}} \\\
S_{\beta} \ar[r]^{\theta_{\beta}} &S_{\beta \pi}}
\end{align}  
 commutes. Then the map $\theta=\bigcup_{\alpha\in Y} \theta_{\alpha}$ is an automorphism of $S$, denoted $\theta=[\theta_{\alpha},\pi]_{\alpha\in Y}$.
\end{theorem}  

We denote the diagram (\ref{1}) by $[\alpha,\beta;\alpha\pi,\beta\pi]$. The map $\pi$ is called the \textit{induced (semilattice) automorphism} of $Y$, denoted $\theta^Y$. 

 Unfortunately, not all automorphisms of strong semilattices of semigroups can be constructed as in Theorem \ref{iso strong}. 
 We shall call a strong semilattice of semigroups $S$ \textit{automorphism-pure} if every automorphism of $S$ can be constructed as in Theorem \ref{iso strong}.  
For example, every strong semilattice of completely simple semigroups is automorphism-pure \cite[Lemma IV.1.8]{Petrich99}, and so both strong semilattices of groups (Clifford semigroups) and strong semilattices of rectangular bands (normal bands) are automorphism-pure.

Let $S=[Y;S_{\alpha};\psi_{\alpha,\beta}]$ be a strong semilattice of semigroups.
 We denote the equivalence relation on $Y$ corresponding to isomorphism types of the semigroups $S_{\alpha}$ by $\eta_S$, so that $\alpha \, \eta_S \, \beta \Leftrightarrow  S_{\alpha}\cong S_{\beta}.$ 
We let $Y^S$ denote the set extension of $Y$ given by $Y^S :=(Y;Y/\eta_S)$.

\begin{proposition}\label{aut pure} Let $S=[Y;S_{\alpha};\psi_{\alpha,\beta}]$ be automorphism-pure and $\aleph_0$-categorical. Then each $S_{\alpha}$ is $\aleph_0$-categorical and $Y^S$ is $\aleph_0$-categorical, with $Y/\eta_S$ finite.
\end{proposition} 

\begin{proof} For each $\alpha\in Y$ fix some $x_{\alpha}\in S_{\alpha}$. We claim that $\{(S_{\alpha},x_{\alpha}):\alpha\in Y\}$ forms a system of  1-pivoted p.r.c. subsemigroups of $S$.
 Indeed, let $\theta$ be an automorphism of $S$ such that $x_{\alpha}\theta=x_{\beta}$ for some $\alpha,\beta\in Y$.
  Since $S$ is automorphism-pure, there exists $\pi\in \text{Aut}(Y)$ and isomorphisms $\theta_{\alpha}\colon S_\alpha \rightarrow S_{\alpha\pi}$ ($\alpha\in Y$) such that $\theta=[\theta_{\alpha},\pi]_{\alpha\in Y}$. Hence $S_{\alpha}\theta=S_{\beta}$, and the claim follows. 
  Consequently, by the $\aleph_0$-categoricity of  $S$ and Proposition \ref{rel-char cat}, each $S_{\alpha}$ is $\aleph_0$-categorical and $Y/\eta_S$ is finite.

 Let $\overline{a}=(\alpha_1,\dots,\alpha_n)$ and $\overline{b}=(\beta_1,\dots,\beta_n)$ be a pair of $n$-tuples of $Y$ such that there exists $a_{\alpha_k}\in S_{\alpha_k}$ and $b_{\beta_k}\in S_{\beta_k}$ with $(a_{\alpha_1}, \dots, a_{\alpha_n}) \, \sim_{S,n} \, (b_{\beta_1},\dots, b_{\beta_n})$ via $[\theta'_{\alpha},\pi']_{\alpha\in Y}\in \text{Aut}(S)$, say. 
Since $\pi'\in \text{Aut}(Y)$ and $S_{\alpha}\cong S_{\alpha\pi'}$ for each $\alpha\in Y$, it follows that $\pi' \in \text{Aut}(Y^S)$.  
Moreover, $\alpha_k\pi'=\beta_k$ for each $k$, so that $\overline{a} \, \sim_{Y^S,n} \, \overline{b}$ via $\pi'$. We have thus shown that 
 \[ |Y^n/\sim_{Y^S,n}|\leq |S^n/\sim_{S,n}|<\aleph_0,
 \]
as $S$ is $\aleph_0$-categorical. Hence $Y^S$ is $\aleph_0$-categorical.   
\end{proof} 

A natural question arises: how can we build an $\aleph_0$-categorical strong semilattice of semigroups from an $\aleph_0$-categorical semilattice and a collection of $\aleph_0$-categorical semigroups? 
In this paper we will only be concerned with the $\aleph_0$-categoricity of strong semilattices of semigroups in which all connecting morphisms are injective or all are constant. 
For arbitrary connecting morphisms, the problem of assessing $\aleph_0$-categoricity appears to be difficult to capture in a reasonable way.
 Examples of more complex $\aleph_0$-categorical strong semilattices of semigroups arise from \cite{Quinnband}, where the \textit{universal} normal band is shown to have surjective but not injective connecting morphisms.   
 We first study the case where each connecting morphism is a constant map. 

 Suppose that $Y$ is a semilattice and, for each $\alpha\in Y$, $S_{\alpha}$ is a semigroup containing an idempotent $e_{\alpha}$. For each $\alpha\in Y$ let $\psi_{\alpha,\alpha}$ be the identity automorphism of $S_{\alpha}$, and for $\alpha> \beta$ let $\psi_{\alpha,\beta}$ be the constant map with image $\{e_{\beta}\}$. 
 We follow the notation of \cite{Worawiset} and let $\psi_{\alpha,\beta}:=C_{\alpha,e_{\beta}}$ for each $\alpha>\beta$ in $Y$. It is easy to check that $\psi_{\alpha,\beta}\psi_{\beta,\gamma}=\psi_{\alpha,\gamma}$ for all $\alpha\geq \beta \geq \gamma$ in $Y$, so that $S=[Y;S_{\alpha};C_{\alpha,e_{\beta}}]$ forms a strong semilattice of semigroups.
  We call $S$ a  \textit{constant strong semilattice of semigroups}.
    
\begin{definition} If $S=[Y;S_{\alpha};C_{\alpha,e_{\beta}}]$ is a constant strong semilattice of semigroups, then we denote the subset of \iso$(S_{\alpha};S_{\beta})$ consisting of those isomorphisms which map $e_{\alpha}$ to $e_{\beta}$ as \iso$(S_{\alpha};S_{\beta})^{[e_{\alpha};e_{\beta}]}$.
 Notice that the set Iso$(S_{\alpha};S_{\alpha})^{[e_{\alpha};e_{\alpha}]}$ is simply the subgroup Aut$(S_{\alpha};\{e_{\alpha}\})$ of Aut($S_{\alpha}$). 
We may then define a relation $\upsilon_S$ on $Y$ by 
\[ \alpha \, \upsilon_S \, \beta \Leftrightarrow \text{\iso}(S_{\alpha};S_{\beta})^{[e_{\alpha};e_{\beta}]}\neq \emptyset,
\] 
so that $\upsilon_S\subseteq \eta_S$.   
\end{definition} 

The relation $\upsilon_S$ is reflexive since $1_{S_{\alpha}}\in \text{Aut}(S_{\alpha};\{e_{\alpha}\})$ for each $\alpha\in Y$, and it easily follows that $\upsilon_S$ forms an equivalence relation on $Y$.  

\begin{proposition}\label{e's cat} Let $S=[Y;S_{\alpha};C_{\alpha,e_{\beta}}]$ be such that $Y/\upsilon_S=\{Y_1,\dots,Y_r\}$ is finite, $\mcal{Y}=(Y;Y_1,\dots,Y_r)$ is $\aleph_0$-categorical and each $S_{\alpha}$ is $\aleph_0$-categorical.
 Then $S$ is $\aleph_0$-categorical.  
\end{proposition} 

\begin{proof} We  prove that $\{S_{\alpha}:\alpha\in Y\}$ forms an $(S;\mcal{Y};\Psi)$-system for some $\Psi$. For each $\alpha,\beta\in Y$, let $\Psi_{\alpha,\beta}=\text{\iso}(S_{\alpha};S_{\beta})^{[e_{\alpha};e_{\beta}]}$ and fix $\Psi=\bigcup_{\alpha,\beta\in Y} \Psi_{\alpha,\beta}$. 
Then Conditions (3.1), (3.2) and (3.3) are seen to be satisfied since $\upsilon_S$ forms an equivalence relation on $Y$. Let $\pi\in \text{Aut}(\mcal{Y})$ and, for each $\alpha\in Y$, let $\theta_{\alpha}\in \Psi_{\alpha,\alpha\pi}$.
We claim that $\theta=[\theta_{\alpha},\pi]_{\alpha\in Y}$ is an automorphism of $S$. 
   Indeed, for any $s_{\alpha}\in S_{\alpha}$ and any $\beta< \alpha$ we have 
\[ s_{\alpha}C_{\alpha,e_{\beta}}\theta_{\beta} = e_{\beta}\theta_{\beta}=e_{\beta\pi} = s_{\alpha}\theta_{\alpha} C_{\alpha\pi,e_{\beta\pi}} 
\] 
so that the diagram $[\alpha,\beta;\alpha\pi,\beta\pi]$ commutes. Moreover $[\alpha,\alpha;\alpha\pi,\alpha\pi]$ commutes as
\[ s_{\alpha}1_{S_{\alpha}}\theta_{\alpha}=s_{\alpha}\theta_{\alpha}=s_{\alpha}\theta_{\alpha}1_{S_{\alpha\pi}},
\] 
and the claim follows by  Theorem \ref{iso strong}. 
Since  $\theta$ extends each $\theta_{\alpha}$, we have that $\{S_{\alpha}:\alpha\in Y\}$ is an $(S;\mcal{Y};\Psi)$-system. 
 Moreover, as $S_\alpha$ is $\aleph_0$-categorical, it is $\aleph_0$-categorical over $\Psi_{\alpha,\alpha}=\text{Aut}(S_{\alpha};\{e_{\alpha}\})$ by \cite[Lemma 2.6]{Quinncat}.
  Hence $S$ is $\aleph_0$-categorical by Corollary \ref{parts hard}.   
\end{proof}

Examining our two main classes of automorphism-pure strong semilattices of semigroups: Clifford semigroups and normal bands, the result above reduces accordingly.
 If $S=[Y;G_{\alpha};C_{\alpha,e_{\beta}}]$ is a constant strong semilattice of groups, then $e_{\alpha}$ is the identity of $G_{\alpha}$, and so \iso $(G_{\alpha};G_{\beta})=\text{\iso}(G_{\alpha};G_{\beta})^{[e_{\alpha};e_{\beta}]}$ for each $\alpha,\beta\in Y$. 
On the other hand, if $S=[Y;B_{\alpha};C_{\alpha,e_{\beta}}]$ is a constant strong semilattice of rectangular bands, then it follows from  Lemma \ref{rb iso} that $\text{\iso}(B_{\alpha};B_{\beta})\neq \emptyset$ if and only if $\text{\iso}(B_{\alpha};B_{\beta})^{[e_{\alpha};e_{\beta}]} \neq \emptyset$, for any $e_{\alpha}\in B_{\alpha},e_{\beta}\in B_{\beta}$.
In both cases we therefore have  $\upsilon_S=\eta_S$.
 Moreover,  each rectangular band $B_{\alpha}$ is $\aleph_0$-categorical by Theorem \ref{RB cat}, and the following result is then immediate by Propositions \ref{aut pure} and \ref{e's cat}. 

\begin{corollary} Let $S=[Y;S_{\alpha};C_{\alpha,e_{\beta}}]$ be a constant strong semilattice of rectangular bands (groups). Then $S$ is $\aleph_0$-categorical if and only if $Y^S$ is $\aleph_0$-categorical, with $Y/\eta_S$ finite  (and each group $S_{\alpha}$ is $\aleph_0$-categorical). 
\end{corollary} 

We now consider the $\aleph_0$-categoricity of a strong semilattice of semigroups $S=[Y;S_{\alpha};\psi_{\alpha,\beta}]$ such that each connecting morphism is injective.
For each $\alpha>\beta$ in $Y$, we abuse notation somewhat by denoting the isomorphism $\psi_{\alpha,\beta}^{-1}|_{\text{Im} \, \psi_{\alpha,\beta}}$ simply by $\psi_{\alpha,\beta}^{-1}$.
 We observe that if $\alpha>\beta>\gamma$ and $x_{\gamma}\in$ Im $\psi_{\alpha,\gamma}$, say $x_{\gamma}=x_{\alpha}\psi_{\alpha,\gamma}$, then 
\[ x_{\gamma} \psi_{\alpha,\gamma}^{-1}\psi_{\alpha,\beta} = x_{\alpha}\psi_{\alpha,\gamma} \psi_{\alpha,\gamma}^{-1}\psi_{\alpha,\beta} = x_{\alpha}\psi_{\alpha,\beta}=x_{\gamma}\psi_{\beta,\gamma}^{-1}. 
\] 
Hence, on the restricted domain Im $\psi_{\alpha,\gamma}$, we have 
\begin{equation} \label{inv conn}
\psi_{\alpha,\gamma}^{-1}\psi_{\alpha,\beta} =  \psi_{\beta,\gamma}^{-1}. 
\end{equation}
If $Y$ has a zero  (i.e. a minimum element under the natural order) we may define an equivalence relation $\xi_S$ on $Y$ by $\alpha \, \xi_S \, \beta$ if and only if $S_{\alpha}\psi_{\alpha,0}=S_\beta\psi_{\beta,0}$. 
If $\alpha \, \xi_S \, \beta$ then $\psi_{\alpha,0}\psi_{\beta,0}^{-1}$ is an isomorphism from $S_{\alpha}$ onto $S_{\beta}$, and so $\xi_S\subseteq \eta_S$.  

\begin{proposition} \label{inj cat} Let $S=[Y;S_{\alpha};\psi_{\alpha,\beta}]$ be such that each $\psi_{\alpha,\beta}$ is injective. 
Let $Y$ be a semilattice with zero and $Y/\xi_S=\{Y_1,\dots,Y_r\}$ be finite, with 
\[ \{S_\alpha\psi_{\alpha,0}:\alpha\in Y\}=\{T_1,\dots,T_r\}.
\]
 Then $S$ is $\aleph_0$-categorical if both $\mcal{Y}=(Y;Y_1,\dots,Y_r)$ and $\mcal{S}_0=(S_0;T_1,\dots,T_r)$ are $\aleph_0$-categorical. 
Moreover, if $S$ is automorphism-pure and $\aleph_0$-categorical, then conversely both $\mcal{Y}$ and $\mcal{S}_0$ are $\aleph_0$-categorical. 
\end{proposition} 

\begin{proof}
Suppose first that both $\mcal{Y}$ and $\mcal{S}_0$ are $\aleph_0$-categorical.  Let $\overline{a}=(a_{\alpha_1},\dots,a_{\alpha_n})$ and $\overline{b}=(b_{\beta_1},\dots,b_{\beta_n})$ be $n$-tuples of $S$ with $(\alpha_1,\dots,\alpha_n) \, \sim_{\mcal{Y},n} \, (\beta_1,\dots,\beta_n)$ via $\pi\in \text{Aut}(\mcal{Y})$, say. 
Suppose further that 
\[ (a_{\alpha_1}\psi_{\alpha_1,0},\dots,a_{\alpha_n}\psi_{\alpha_n,0}) \, \sim_{\mcal{S}_0,n} \, (b_{\beta_1}\psi_{\beta_1,0},\dots,b_{\beta_n}\psi_{\beta_n,0})
\] 
 via $\theta_0\in \text{Aut}(\mcal{S}_0)$, say. 
 Then for each $\alpha\in Y$ we have $S_\alpha\psi_{\alpha,0}=S_{\alpha\pi}\psi_{\alpha\pi,0}$, and so we can take an isomorphism $\theta_{\alpha}\colon S_{\alpha}\rightarrow S_{\alpha\pi}$ given by 
\[ \theta_\alpha=\psi_{\alpha,0} \, \theta_0 \, \psi_{\alpha\pi,0}^{-1}.
\] 
For each $\alpha\geq \beta$ in $Y$, the diagram $[\alpha,\beta;\alpha\pi,\beta\pi]$ commutes as  
\begin{align*}
  \psi_{\alpha,\beta} \,  {\theta}_{\beta} & =   \psi_{\alpha,\beta} \, ( \psi_{\beta,0} \, \theta_0 \, \psi_{\beta\pi,0}^{-1}) =    {\psi}_{\alpha,0} \, {\theta}_0 \, \psi_{\beta\pi,0}^{-1} \\
 & =    {\psi}_{\alpha,0} \, {\theta}_0 \, ({\psi}_{\alpha\pi,0}^{-1} \,  {\psi}_{\alpha\pi,\beta\pi}) = {\theta}_{\alpha} \, \psi_{\alpha\pi,\beta\pi},
 \end{align*}  
where the penultimate equality is due to \eqref{inv conn} as Im $\psi_{\alpha\pi,0} =$ Im $\psi_{\alpha,0} = $ (Im $\psi_{\alpha,0})\theta_0$. 
Hence $\theta=[\theta_{\alpha},\pi]_{\alpha\in Y}$ is an automorphism of $S$ by Theorem \ref{iso strong}. 
Furthermore, 
\[ a_{\alpha_k}\theta=a_{\alpha_k}\theta_{\alpha_k} = a_{\alpha_k} \psi_{\alpha_k,0} \, \theta_0 \, \psi_{\alpha_k\pi,0}^{-1} = b_{\beta_k} \psi_{\beta_k,0} \, \psi_{\beta_k,0}^{-1} = b_{\beta_k} 
\] 
for each $1\leq k \leq n$, so that $\overline{a} \, \sim_{S,n} \, \overline{b}$ via $ \theta$. We thus have that 
\[ |S^n/\sim_{S,n}| \leq |\mcal{Y}^n/\sim_{\mcal{Y},n}| \cdot |\mcal{S}_0^n/\sim_{\mcal{S}_0,n}|<\aleph_0
\] 
and so $S$ is $\aleph_0$-categorical. 

Conversely, suppose $S$ is automorphism-pure and $\aleph_0$-categorical.  For each $1\leq k \leq r$, fix some $\gamma_k\in  Y_k$, where we assume without loss of generality that $S_{\gamma_k}\psi_{\gamma_k,0}=T_k$. 
For each $\alpha\in Y$, fix some $x_{\alpha}\in S_{\alpha}$. Let $\overline{a}=(\alpha_1,\dots,\alpha_n)$ and $\overline{b}=(\beta_1,\dots,\beta_n)$  be $n$-tuples of $Y$ such that 
\[ (x_{\alpha_1},\dots,x_{\alpha_n}, x_{\gamma_1},\dots,x_{\gamma_r}) \, \sim_{S,n+r}  \, (x_{\beta_1},\dots,x_{\beta_n}, x_{\gamma_1},\dots,x_{\gamma_r}),
\] 
via $\theta\in \text{Aut}(S)$, say. 
Since $S$ is automorphism-pure there exists $\pi\in \text{Aut}(Y)$ and $\theta_{\alpha}\in \text{Iso}(S_{\alpha};S_{\alpha\pi})$ such that $\theta=[\theta_{\alpha},\pi]_{\alpha\in Y}$. 
The automorphism $\pi$ fixes each $\gamma_k$, so that $S_{\gamma_k}\theta=S_{\gamma_k}$. 
Hence, as the diagram $[\gamma_k,0;\gamma_k,0]$ commutes for each $k$, we have
 \[ T_k = S_{\gamma_k}\psi_{\gamma_k,0}=(S_{\gamma_k}\theta_{\gamma_k})\psi_{\gamma_k,0}=S_{\gamma_k}\psi_{\gamma_k,0}\theta_0=T_k\theta_0=T_k\theta.
 \] 
If $\alpha\in Y_k$ then, by the commutativity of the diagram $[\alpha;0;\alpha\pi,0]$, we therefore have 
\[ S_{\alpha}\psi_{\alpha,0} = T_k= T_k\theta_0=S_{\alpha}\psi_{\alpha,0}\theta_0 = S_{\alpha}\theta_{\alpha}\psi_{\alpha\pi,0} = S_{\alpha\pi}\psi_{\alpha\pi,0}, 
\]
 and so $\pi\in \text{Aut}(\mcal{Y})$. 
 We have shown that 
\[ |\mcal{Y}^n/\sim_{\mcal{Y},n}| \leq |S^{n+r}/\sim_{S,n+r}|<\aleph_0
\] 
and so $\mcal{Y}$ is $\aleph_0$-categorical. 
Now suppose $\overline{c}$ and $\overline{d}$ are $n$-tuples of $\mcal{S}_0$ such that 
\[ (\overline{c},x_{\gamma_1},\dots,x_{\gamma_r}) \, \sim_{S,n+r} \, (\overline{d},x_{\gamma_1},\dots,x_{\gamma_r}),  
\] 
via $\theta'=[\theta_{\alpha}',\pi']_{\alpha\in Y}\in \text{Aut}(S)$, say. 
Then arguing as before we have that $T_k\theta'=T_k$ for each $k$, and it follows that  $\theta_0'\in \text{Aut}(\mcal{S}_0)$, with $\overline{c}\theta_0'=\overline{d}$. Hence 
\[ |\mcal{S}_0^n/\sim_{\mcal{S}_0,n}|\leq |S^{n+r}/\sim_{S,n+r}|<\aleph_0
\] 
and so $\mcal{S}_0$ is $\aleph_0$-categorical.  
\end{proof}

Note that if $Y$ is finite, then the meet of all the elements of $Y$ is a zero. Moreover, as $Y$ is finite, it is $\aleph_0$-categorical over any set of subsets by the RNT, and so the result above simplifies accordingly in this case: 

 \begin{corollary}  Let $S=[Y;S_{\alpha};\psi_{\alpha,\beta}]$ be such that $Y$ is finite and each $\psi_{\alpha,\beta}$ is injective. 
 If  $\mcal{S}_0=({S}_0;\{S_{\alpha}\psi_{\alpha,0}:\alpha\in Y\})$ is $\aleph_0$-categorical  then $S$ is $\aleph_0$-categorical. 
Conversely, if $S$ is automorphism-pure and $\aleph_0$-categorical then $\mcal{S}_0$ is $\aleph_0$-categorical.
\end{corollary} 

For a Clifford semigroup $S$, the property that the connecting morphisms are injective is equivalent to $S$ being is $E$-unitary, that is, such that for all $e\in E(S)$ and all $s\in S$, if $es\in E$ then $s\in E(S)$ \cite[Exercise 5.20]{Howie94}. Since Clifford semigroups are automorphism-pure, we therefore have the following simplification of Proposition \ref{inj cat}. 

\begin{corollary}\label{E-un} Let $S=[Y;G_{\alpha};\psi_{\alpha,\beta}]$ be an $E$-unitary Clifford semigroup. Let $Y$ be a semilattice with zero and $Y/\xi_S$ be finite.  
Then $S$ is $\aleph_0$-categorical if and only if  $(Y;Y/\xi_S)$  and $(S_0; \{S_\alpha\psi_{\alpha,0}:\alpha\in Y\})$ are $\aleph_0$-categorical. 
In particular, if $Y$ is finite then $S$ is $\aleph_0$-categorical if and only if $(S_0; \{S_\alpha\psi_{\alpha,0}:\alpha\in Y\})$ is $\aleph_0$-categorical. 
\end{corollary} 

\begin{example} We use the work of Apps \cite{Apps82} to construct examples of  $\aleph_0$-categorical $E$-unitary Clifford semigroups as follows. 
Let $G$ be an $\aleph_0$-categorical group and $H_1<H_2<\cdots$ a characteristic series in $G$, so that each $H_i$ is a characteristic subgroup of $G$ and $H_i$ is a subgroup of $H_{i+1}$. 
Apps proved that such a series must be finite, and there exists a characteristic series $\{1\}=G_0<G_1<G_2<\cdots <G_n=G$ with each $G_i/G_{i-1}$ a characteristically simple $\aleph_0$-categorical group. 
For each $0\leq i \leq n$, let $K_i=G_i\times \{i\}$ be an isomorphic copy of $G_i$.  
For each ${0}\leq {i}  \leq j \leq  {n}$, let $\psi_{i,j}\colon K_i\rightarrow K_{j}$ be the map given by $(x,i)\psi_{i,j}=(x,j)$. 
Then we may form a strong semilattice of the groups $K_i$ by taking $S=[Y;K_i;\psi_{i,j}]$, where $Y$ is the set $\{0, 1,\dots,n\}$ with the reverse ordering $0>1>2>\cdots >n$. 
Notice that $S$ is $E$-unitary as each connecting morphism is injective. Moreover, each $K_i\psi_{i,n}=G_i\times \{n\}$ is a characteristic subgroup of $K_n=G_n\times \{n\}$. 
Hence, by Lemma \ref{over char 2}, $(K_n;\{K_i\psi_{i,n}:1\leq i \leq n \})$ is $\aleph_0$-categorical. 
Since $Y$ is finite, we have that $(Y;Y/\xi_S)$ is $\aleph_0$-categorical, and so $S$ is $\aleph_0$-categorical by Corollary \ref{E-un}. 
\end{example} 

If $S=[Y;S_{\alpha};\psi_{\alpha,\beta}]$ is such that each connecting morphism is an isomorphism, then $Y/\xi_S=\{Y\}$, and so the result above simplifies accordingly. 
However we can  prove a more general result directly (without the condition that $Y$ has a zero) with aid of the following proposition. The result is folklore, but a proof can be found in \cite{Quinn}. 

\begin{proposition}\label{iso state} Let $S=[Y;S_{\alpha};\psi_{\alpha,\beta}]$ be such that each $\psi_{\alpha,\beta}$ is an isomorphism. 
Then  $S\cong S_{\alpha} \times Y$ for any $\alpha\in Y$.  
Conversely, if $T$ is a semigroup and $Z$ is a semilattice then $T\times Z$ is isomorphic to a strong semilattice of semigroups such that each connecting morphism is an isomorphism. 
\end{proposition}

\begin{corollary} Let $S=[Y;S_{\alpha};\psi_{\alpha,\beta}]$ be such that each $\psi_{\alpha,\beta}$ is an isomorphism. 
If $S_{\alpha}$ and $Y$ are $\aleph_0$-categorical, then $S$ is $\aleph_0$-categorical. Moreover, if $S$ is automorphism-pure then the converse holds. 
\end{corollary} 

\begin{proof} By Proposition \ref{iso state}, $S$ is isomorphic to $S_{\alpha}\times Y$ for any $\alpha\in Y$. 
The first half of the result then follows as $\aleph_0$-categoricity is preserved by finite direct products \cite{Grzeg}. 

If $S$ is automorphism-pure then the converse holds by Proposition \ref{aut pure}, as $(Y;Y/\eta_S)$ being $\aleph_0$-categorical clearly implies $Y$ is $\aleph_0$-categorical.   
\end{proof}

\section{$\aleph_0$-categorical Rees matrix semigroups} \label{Sec rees} 

 A semigroup $S$ is called simple (0-simple) if it has no proper ideals (if its only proper ideal is $\{0\}$ and $S^2\neq \{0\}$). 
 A simple (0-simple) semigroup is called \textit{completely simple} (\textit{completely 0-simple}) if contains a primitive idempotent, i.e. a non-zero idempotent $e$ such that for any non-zero idempotent $f$ of $S$, 
\[ ef=fe=f \Rightarrow e=f. 
\] 
Since an $\aleph_0$-categorical semigroup is periodic, it follows that every $\aleph_0$-categorical (0-)simple semigroup is completely (0)-simple (see the proof of Theorem 3.12 of \cite{Quinncat}). 
By Rees Theorem \cite{Rees}, to study the $\aleph_0$-categoricity of a completely 0-simple semigroup, it is sufficient to consider Rees matrix semigroups: 

\begin{theorem}[The Rees Theorem] Let $G$ be a group, let $I$ and $\Lambda$ be non-empty index sets and let $P=(p_{\lambda,i})$ be an $\Lambda \times I$ matrix with entries in $G\cup \{0\}$. 
Suppose no row or column of $P$ consists entirely of zeros (that is, $P$ is \textit{regular}). Let $S=(I\times G \times \Lambda) \cup \{0\}$, and define multiplication $*$ on $S$ by 
\begin{align*}  & (i,g,\lambda)*(j,h,\mu) = \left\{
\begin{array}{ll}
(i,g p_{\lambda, j} h,\mu)  & \text{if       }  p_{\lambda, j}\neq 0\\
0 & \text{else }  
\end{array}
\right. \\
& 0*(i,g,\lambda) =(i,g,\lambda)*0=0*0=0.
\end{align*} 
Then $S$ is a completely 0-simple semigroup, denoted $\mcal{M}^0[G;I,\Lambda;P]$, and is called a (regular) Rees matrix semigroup (over $G$). 
Conversely, every completely 0-simple semigroup is isomorphic to a Rees matrix semigroup.  
\end{theorem} 

The matrix $P$ is called the \textit{sandwich matrix} of $S$. If $P$ has no zero entries, then $I\times G \times \Lambda$ forms a subsemigroup of $\mcal{M}^0[G;I,\Lambda,P]$, called a \textit{Rees matrix semigroup without zero} and denoted $\mcal{M}[G;I,\Lambda;P]$. 
Every  completely simple semigroup is isomorphic to a Rees matrix semigroup without zero \cite[Section 3.3]{Howie94}.

\begin{lemma} Let $G$ be a group and $P$ be a $\Lambda\times I$ matrix with entries from $G$. 
Then $\mcal{M}[G;I,\Lambda;P]$ is $\aleph_0$-categorical if and only if $\mcal{M}^0[G;I,\Lambda;P]$ is $\aleph_0$-categorical. 
\end{lemma}

\begin{proof} The result is immediate from \cite[Corollary 2.12]{Quinncat} since $\mcal{M}^0[G;I,\Lambda;P]$ is isomorphic to $\mcal{M}[G;I,\Lambda;P]$ with a zero adjoined.  
\end{proof} 

As a consequence, to examine the $\aleph_0$-categoricity of both completely simple and completely 0-simple semigroups, it suffices to study   Rees matrix semigroups. 

A fundamental discovery in \cite{Quinncat} was that to understand the $\aleph_0$-categoricity of an arbitrary semigroup, it is necessary to study $\aleph_0$-categorical completely (0-)simple semigroups. Indeed, they arise as principal factors of an $\aleph_0$-categorical semigroup, as well as giving examples of 0-direct indecomposable summands in a semigroup with zero. 

In \cite{Quinncat} the $\aleph_0$-categoricity of Rees matrix semigroups over identity matrices (known as \textit{Brandt semigroups}) were determined, although we deferred the general case to this current article. 
Countable homogeneous completely simple semigroups have been  classified (modulo our understanding of homogeneous groups) in \cite{Quinncss}, which gives rise to more complex examples of $\aleph_0$-categorical completely (0-)simple semigroups. 

Given a Rees matrix semigroup $S=\mcal{M}^0[G;I,\Lambda;P]$ with $P=(p_{\lambda,i})$, we let $G(P)$ denote the subset of $G$ of all non-zero entries of $P$, that is,  $ G(P): = \{p_{\lambda,i}: p_{\lambda,i}\neq 0\}$. 
The idempotents of $S$ are easily described \cite[Page 71]{Howie94}: 
\[ E(S) = \{(i,p_{\lambda,i}^{-1},\lambda):p_{\lambda,i}\neq 0\}. 
\] 

Since there exists a  simple isomorphism theorem for Rees matrix semigroups  \cite[Theorem 3.4.1]{Howie94} (see Theorem \ref{iso c0s}), we should be hopeful of achieving a thorough understanding of  $\aleph_0$-categorical Rees matrix semigroups via the RNT. 
 However, from the isomorphism theorem it is not clear how the $\aleph_0$-categoricity of the semigroup $\mcal{M}^0[G;I,\Lambda;P]$ affects the sets $I$ and $\Lambda$.
  We instead follow a technique of Graham \cite{Graham68} and Houghton \cite{Houghton77} of constructing a bipartite graph from the sets $I$ and $\Lambda$.
  
A \textit{bipartite graph} is a (simple) graph whose vertices can be split into two disjoint non-empty sets $L$ and $R$ such that every edge connects a vertex in $L$ to a vertex in $R$.
 The sets $L$ and $R$ are called the \textit{left set} and the \textit{right set}, respectively. 
 Formally, a bipartite graph is a triple $\Gamma=\langle L,R, E \rangle$ such that $L$ and $R$ are non-empty trivially intersecting sets and 
\[ E\subseteq \{ \{x,y\}  :x\in L, \, y\in R \}.
\] 
 We call $L\cup R$ the set of \textit{vertices} of $\Gamma$ and $E$ the set of \textit{edges}. 
 An isomorphism between a pair of bipartite graphs $\Gamma=\langle L,R,E \rangle$ and $\Gamma'=\langle L',R',E' \rangle$ is a bijection $\psi\colon  L\cup R \rightarrow L'\cup R'$ such that $L\psi=L'$,  $R\psi=R'$, and  $\{l,r\}\in E$ if and only if $\{l\psi,r\psi\}\in E'$. 
 We are therefore regarding bipartite graphs in the signature $L_{BG}=\{Q_L,Q_R,E\}$, where $Q_L$ and $Q_R$ are unary relations, which  correspond to the sets $L$ and $R$, respectively, and $E$ is a binary relation corresponding to the edge relation (here we abuse the notation somewhat by letting $E$ denote the edge relation \textit{and} the set of edges). 

Let $\Gamma=\langle L,R, E \rangle$ be a bipartite graph.
 Then $\Gamma$ is called \textit{complete} if,  for all $x\in L, \, y\in R$, we have $\{x,y\}\in E$. If $E=\emptyset$ then $\Gamma$ is called  \textit{empty}. 
 If each vertex of $\Gamma$ is incident to exactly one edge, then $\Gamma$ is called a \textit{perfect matching}. 
 The \textit{complement} of $\Gamma$ is the bipartite graph $\langle L, R,E'\rangle$ with 
\[ E'=\{\{x,y\}: x\in L, y\in R, \{x,y\}\not\in E\}.
\] 
 Hence an empty bipartite graph is the complement of a complete bipartite graph, and vice-versa.  We call $\Gamma$ \textit{random} if, for each $k,\ell \in \mathbb{N}$, and for every distinct $x_1,\dots,x_k,y_1,\dots,y_{\ell}$ in $L$ (in $R$) there exists infinitely many $u\in R$ ($u\in L$) such that $\{u,x_i\}\in E$ but $\{u,y_j\}\not\in E$ for each $1\leq i \leq k$ and $1 \leq j \leq \ell$.

Clearly, for each pair $n,m\in \mathbb{N}^*=\mathbb{N}\cup\{\aleph_0\}$, there exists a unique (up to isomorphism) complete biparite graph with left set of size $n$ and right set of size $m$,  which we denote as $K_{n,m}$. 
There also exists a unique, up to isomorphism, perfect matching with left and right sets of size $n$, denoted $P_n$. 
Similar uniqueness  holds for the empty bipartite graph $E_{n,m}$ with left set of size $n$ and right set of size $m$, and the complement of the perfect matching $P_n$, which we denote as $CP_n$. Less obviously, any pair of random bipartite graphs are isomorphic \cite{Erdos}.

\begin{theorem}\label{BG CAT LIST} \cite{Goldstern96} A countable bipartite graph is homogeneous if and only if it is isomorphic to either $K_{n,m}$, $E_{n,m}$, $P_n$, $CP_n$ for some $n,m\in \mathbb{N}^*$, or the  random bipartite graph.
\end{theorem} 

Since bipartite graphs are relational structures with finitely many relations,  homogeneous bipartite graphs are uniformly locally finite\footnote{This differs from the graph theoretical notion of being uniformly locally finite, i.e. such that the degrees of the vertices are bounded above by some finite value.}, and thus $\aleph_0$-categorical. 
Unfortunately, no full classification of $\aleph_0$-categorical bipartite graphs exists. 

 Let $\Gamma=\langle L,R, E \rangle$ be a bipartite graph. A \textit{path} $\mathfrak{p}$ in $\Gamma$ is a finite sequence of vertices 
 \[ \mathfrak{p}=(v_0,v_1,\dots, v_n)
 \] 
 such that $v_i$ and $v_{i+1}$ are adjacent for each $0\leq i \leq n-1$. 
 For example, if $\{x,y\}$ is an edge in $E$ then both $(x,y)$ and $(y,x)$ are paths in $\Gamma$. A pair of vertices $x$ and $y$ are \textit{connected}, denoted $x \Join y$, if and only if $x=y$ or there exists a path $(v_1,v_2,\dots, v_n)$ in $\Gamma$ such that $v_1=x$ and $v_n=y$. 
 It is clear that $\Join$ is an equivalence relation on the set of vertices of $\Gamma$, and we call the equivalence classes the \textit{connected components} of $\Gamma$. 
 Each connected component is a sub-bipartite graph of $\Gamma$ under the induced structure, and we let $\mcal{C}(\Gamma)$ denote the set of connected components of $\Gamma$.

 Let $\Gamma$ be a bipartite graph with $\mcal{C}(\Gamma) =\{\Gamma_i:i\in A\}$. For any automorphism $\phi$ of $\Gamma$ and $x,y\in \Gamma$ we have that $(x,v_2,\dots,v_{n-1},y)$ is a path in $\Gamma$ if and only if $(x\phi,v_2\phi,\dots,v_{n-1}\phi,y\phi)$ is a path in $\Gamma$, since $\phi$ preserves edges and non-edges. 
 Hence $x \, \Join \, y$ if and only if $x\phi \, \Join \, y\phi$, and so there exists a bijection $\pi$ of $A$ such that $\Gamma_i\phi=\Gamma_{i\pi}$ for each $i\in I$. 
 We have thus proven the reverse direction of the following result, the  forward being immediate.
 
 \begin{proposition}\label{auto BG} Let $\Gamma=\langle L,R, E \rangle$ be a bipartite graph with $\mcal{C}(\Gamma)=\{\Gamma_i:i\in A\}$.
  Let $\pi$ be a bijection of $A$ and $\phi_i\colon \Gamma_i\rightarrow \Gamma_{i\pi}$ an isomorphism for each $i\in A$. Then $\bigcup_{i\in I} \phi_i$ is an automorphism of $\Gamma$. Conversely, every automorphism of $\Gamma$ can be constructed in this way. 
 \end{proposition} 
 
 \begin{proposition}\label{bg iff cc}  Let $\Gamma=\langle L,R, E \rangle$ be a bipartite graph with  $\mcal{C}(\Gamma)=\{\Gamma_i:i\in A\}$. 
 Then $\Gamma$ is $\aleph_0$-categorical if and only if each connected component is $\aleph_0$-categorical and $\mcal{C}(\Gamma)$ is finite, up to isomorphism.
 \end{proposition} 

\begin{proof} 
($\Rightarrow$) By Proposition \ref{auto BG} we have that, for any choice of $x_i\in \Gamma_i$ ($i\in A$), the set $\{(\Gamma_i,x_i):i\in A\}$ forms a system of 1-pivoted p.r.c. sub-bipartite graphs of $\Gamma$. The result then follows from Proposition \ref{rel-char cat}. 

($\Leftarrow$) First we show that $\mcal{C}(\Gamma)$ forms a $(\Gamma;\underline{A};\Psi)$-system in $\Gamma$ for some $\underline{A}$ and $\Psi$. 
Let $A_1,\dots,A_r$ be the finite partition of $A$ corresponding to the isomorphism types of the connected components of $\Gamma$, that is, $\Gamma_i\cong  \Gamma_j$ if and only if $ i,j\in A_k$ for some $k$. 
Fix $\underline{A}=(A;A_1,\dots,A_r)$. For each $i,j\in A$, let $\Psi_{i,j}=\text{\iso}(\Gamma_i;\Gamma_j)$ and fix $\Psi=\bigcup_{i,j\in A} \Psi_{i,j}$. 
Then $\Psi$ clearly satisfy Conditions (A), (B) and (C). Let $\pi\in \text{Aut}(\underline{A})$ and, for each $i\in A$, let $\phi_i\in \Psi_{i,i\pi}$. 
Then by Proposition \ref{auto BG}, $\phi=\bigcup_{i\in A} \phi_i$ is an automorphism of $\Gamma$, and so $\Psi$ satisfies Condition (D). 
Hence $\mcal{C}(\Gamma)$ forms an $(\Gamma;\underline{A};\Psi)$-system. Each $\Gamma_i$ is $\aleph_0$-categorical (over $\Psi_{i,i}=\text{Aut}(\Gamma_i)$) and  $\underline{A}$ is $\aleph_0$-categorical by Corollary \ref{set plus cat},  and so $\Gamma$ is $\aleph_0$-categorical by Corollary \ref{parts hard}.  
\end{proof}

  \begin{definition} Let $S=\mcal{M}^0[G;I,\Lambda; P]$ be a Rees matrix semigroup with $P=(p_{\lambda,i})$. 
  Then we form a bipartite graph $\Gamma(P)=\langle I, \Lambda, E\rangle$ with edge set 
\[ E=\{\{i,\lambda\}: p_{\lambda,i} \neq 0\},
\] 
 which we call the \textit{induced bipartite graph of $ S$}. 
\end{definition}

The above construct has long been fundamental to the study of Rees matrix semigroups, and has its roots in a paper by Graham in \cite{Graham68}. 
Here, it is used to describe the maximal nilpotent subsemigroups of a Rees matrix semigroup, where a semigroup is \textit{nilpotent} if some power is equal to $\{0\}$.
 All maximal subsemigroups of a finite Rees matrix semigroup were described in the same paper, a result which was later extended in \cite{Grahams} to arbitrary finite semigroups.
  In \cite{Howie76}, Howie used the induced bipartite graph to describe the subsemigroup of a Rees matrix semigroup  generated by its idempotents.
   Finally, in \cite{Houghton77}, Houghton described the homology of the induced bipartite graph, and a detailed overview of his work is given in \cite{Rhodes}.   

 \begin{example} Let $S=\mcal{M}^0[G;\{1,2,3\},\{\lambda,\mu\}; P]$ where 
\[ P= 
 \kbordermatrix{ & 1 & 2 & 3 \\
      \lambda & a & b & 0 \\
      \mu & 0 & c & d  }. 
      \] 
Then the induced bipartite graph of $S$ is given in Figure \ref{induced S}.   
\begin{figure}[h]
\centering
\def\svgwidth{170pt} 
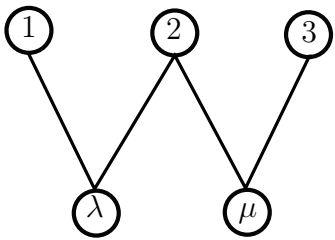 
\caption{Induced bipartite graph} 
\label{induced S} 
\end{figure} 
\end{example}

\begin{example} Let $S=\mcal{M}^0[G;I,\Lambda;P]$ be such that $P$ has no zero entries, so that $S$ is isomorphic to a completely simple semigroup with zero adjoined. 
Then $\Gamma(P)$ is a complete bipartite graph. 
\end{example}  

\begin{notation}\label{tuple not} Let $S=\mcal{M}^0[G;I,\Lambda;P]$ be a Rees matrix semigroup. For an $n$-tuple  $\overline{a}=((i_1,g_1,\lambda_1),\dots,(i_n,g_n,\lambda_n))$ of $S^*$,  we write $\Gamma(\overline{a})$ for the $2n$-tuple $(i_1,\lambda_1,\dots,i_n,\lambda_n)$ of $ \Gamma(P)$. 
\end{notation} 

Following \cite{Araujo10}, we adapt the isomorphism theorem for Rees matrix semigroups to explicitly highlight the roll of the induced bipartite graph: 

\begin{theorem} \label{iso c0s} Let $S_1=\mcal{M}^0[G_1; I_1, \Lambda_1;P_1]$ and $S_2=\mcal{M}^0[G_2;I_2,\Lambda_2;P_2]$ be a pair of Rees matrix semigroups with sandwich matrices $P_1=(p_{\lambda, i})$ and $P_2=(q_{\mu, j})$, respectively. 
Let $\psi\in \text{Iso}(\Gamma(P_1);\Gamma(P_2))$,  $\theta\in \text{Iso}(G_1;G_2)$, and $u_i, v_{\lambda}\in G_2$ for each $i\in I_1, \lambda \in \Lambda_1$. 
Then the mapping $\phi\colon S_1\rightarrow S_2$ given by 
\[ (i,g,\lambda)\phi= (i\psi, u_i  (g\theta)  v_{\lambda}, \lambda \psi)
\]
is an isomorphism if and only if $p_{\lambda, i} \, \theta = v_{\lambda} \cdot q_{\lambda \psi, i\psi} \cdot u_i$  whenever $p_{\lambda, i}\neq 0$. Moreover, every isomorphism from $S_1$ to $S_2$ can be described in this way. 
\end{theorem} 

The isomorphism $\phi$ will be denoted as $(\theta,\psi,(u_i)_{i\in I},(v_{\lambda})_{\lambda\in \Lambda})$. 
We also denote the induced group isomorphism $\theta$ as $\phi_{G_1}$, and the induced bipartite graph isomorphism $\psi$  as $\phi_{\Gamma(P_1)}$, so that $\phi=(\phi_{G_1},\psi_{\Gamma(P_1)},(u_i)_{i\in I_1},(v_{\lambda})_{\lambda\in \Lambda_1})$. 
Note that the induced group isomorphism is not uniquely defined by $\phi$. That is, there may exist $\theta'\in \text{Iso}(G_1;G_2)$ and $u_i',v_{\lambda}'\in G_2$, such that $\theta'\neq \theta$ but $\phi=(\theta',\psi,(u_i')_{i\in I_1},(v_{\lambda}')_{\lambda\in \Lambda_1})$. Examples of this phenomenon will occur throughout this work. 

The composition and inverses of isomorphisms between Rees matrix semigroups behave in a natural way as follows, and a proof can be found in \cite{Quinn}. 

\begin{corollary}\label{induced iso} Let $S_k=\mathcal{M}^0[G_k;I_k,\Lambda_k;P_k]$ ($k=1,2,3$) be Rees matrix semigroups. 
Then for any pair of isomorphisms $\phi=(\theta,\psi,(u_i)_{i\in I_1},(v_{\lambda})_{\lambda\in \Lambda_1})\in \text{\iso}(S_1;S_2)$ and $\phi'=(\theta',\psi',(u_j')_{j\in I_2},(v_{\mu}')_{\mu\in \Lambda_2})\in \text{\iso}(S_2;S_3)$ we have:
\begin{enumerate}[label=(\roman*)]
 \item  $\phi\phi'=\big(\theta\theta',\psi\psi',(u'_{i\psi}(u_i\theta'))_{i\in I_1},((v_{\lambda}\theta')v_{\lambda\psi}')_{\lambda\in \Lambda_1} \big)$; 
\item $\phi^{-1}=(\theta^{-1},\psi^{-1},((u_{i\psi^{-1}})^{-1}\theta^{-1})_{i\in I_2},((v_{\lambda\psi^{-1}})^{-1}\theta^{-1})_{\lambda\in \Lambda_2})$. 
\end{enumerate} 
\end{corollary} 

Let $\Gamma=\langle L,R,E \rangle$ be a bipartite graph. For each $n\in \mathbb{N}$, we let  $\sigma_{\Gamma,n}$ be the equivalence relation  on $\Gamma^n$ given by 
\[ (x_1,\dots,x_n) \, \sigma_{\Gamma,n} \, (y_1,\dots,y_n) \Leftrightarrow [x_i\in L \Leftrightarrow y_i\in L, \text{ for each } 1\leq i \leq n].
\]
Since each entry of an $n$-tuple of $\Gamma$ lies in either $L$ or $R$ we have that
\[ |\Gamma^n/\sigma_{\Gamma,n}|=2^n, 
\] 
for each $n$. 
Moreover, as the automorphisms of $\Gamma$ fixes the sets $L$ and $R$, it easily follows that $\sim_{\Gamma,n} \,  \subseteq \,  \sigma_{\Gamma,n}$.

\begin{proposition}\label{G and gamma cat} If $S= \mcal{M}^0[G; I, \Lambda;P]$ is $\aleph_0$-categorical, then $G$ and $\Gamma(P)$ are $\aleph_0$-categorical. 
\end{proposition} 

\begin{proof}
Since $G$ is isomorphic to the non-zero maximal subgroups of $S$, it is $\aleph_0$-categorical by \cite[Corollary 3.7]{Quinncat}.
 Now let $\overline{a}=(a_1,\dots,a_n)$ and $\overline{b}=(b_1,\dots,b_n)$ be a pair of $\sigma_{\Gamma(P),n}$-related $n$-tuples of $\Gamma(P)$. 
 Let $i_1<i_2<\cdots<i_s$ and  $j_1<j_2<\cdots<j_t$ be the indexes of entries of $\overline{a}$ lying in $I$ and $\Lambda$, respectively (noting that the same is true for $\overline{b}$ as $\overline{a} \, \sigma_{\Gamma,n} \, \overline{b}$).
Suppose further that there exists $i\in I,\lambda\in \Lambda$ such that the $n$-tuples 
\begin{align*} & ((a_{i_1},1,\lambda), \dots, (a_{i_s},1,\lambda),(i,1,a_{j_1}),\dots,(i,1,a_{j_t})) \quad \text{and} \\
& ((b_{i_1},1,\lambda), \dots, (b_{i_s},1,\lambda),(i,1,b_{j_1}),\dots,(i,1,b_{j_t})),
\end{align*}
are automorphically equivalent via $\phi\in \text{Aut}(S)$, say.
 By Theorem \ref{iso c0s},  $a_{i_r}\phi_{\Gamma(P)}=b_{i_r}$ and $a_{j_{r'}}\phi_{\Gamma(P)}=b_{j_{r'}}$ for each $1\leq r \leq s$ and $1\leq r'\leq t$. 
 Hence $\overline{a} \, \sim_{\Gamma(P),n} \, \overline{b}$ via $\phi_{\Gamma(P)}$, and we have thus shown that 
\[ |\Gamma(P)^{n}/\sim_{\Gamma(P),n}| \leq 2^{n} \cdot |S^n/\sim_{S,n}|. 
\] 
Hence $\Gamma(P)$ is $\aleph_0$-categorical by the $\aleph_0$-categoricity of $S$.  
\end{proof}

However, the converse to the proposition above does not hold in general (even in the completely simple case).   

\begin{example}\label{example main} Let $G=\{1,a\}$ be the group of size 2 and let $I=\{i_0,i_1,\dots,\}$ and $\Lambda=\{\lambda_0,\lambda_1,\dots,\}$ be infinite sets. 
Let $P$ be the $\Lambda\times I$ matrix in which    $p_{\lambda_k,i_{\ell}}=a$ if and only if $k\geq \ell \geq 1$, that is, 
 \[  P  = \begin{bmatrix}
 1 & 1 & 1 & 1 & \cdots & 1 & \cdots \\
 1& a & 1 & 1 & \cdots & 1 & \cdots \\
 1 & a & a & 1 & \ddots & \ddots & \vdots &  \\
 \vdots & \vdots & \ddots & \ddots & 1 & 1 & \cdots \\
 1 & a & \cdots & a & a & 1  & \cdots \\
 1 & a& \cdots & a & a & a & \cdots \\
 \vdots & \vdots & \cdots & \vdots & \vdots & \vdots & \ddots 
 \end{bmatrix}.\]
Let $S=\mcal{M}[G;I,\Lambda;P]$. Then $\Gamma(P)$ is a complete bipartite graph, and thus $\aleph_0$-categorical. 
However, $\{((i_0,1,\lambda_0),(i_k,1,\lambda_k)):k\in \mathbb{N}\}$ can be shown to be an infinite set of distinct 2-automorphism types of $S$. 
Alternatively, we will show at the end of the section that $S$ is not $\aleph_0$-categorical by Proposition \ref{prop:css}. 
\end{example} 

\subsection{Connected Rees components} 

Let $S_k=\mcal{M}^0[G;I_k,\Lambda_k;P_k]$ ($k\in A$) be a collection of Rees matrix semigroups with $P_k=(p_{\lambda,i}^{(k)})$ and $S_k\cap S_{\ell}=\{0\}$ for eack $k,\ell\in A$. 
Then we may form a single Rees matrix semigroup $S=\mcal{M}^0[G;I,\Lambda;P]$, where $I=\bigcup_{k\in A} I_k$, $\Lambda=\bigcup_{k\in A} \Lambda_k$ and $P=(p_{\lambda,i})$ is the $\Lambda$ by $I$ matrix defined by 
\[ p_{\lambda,i} = \left\{
\begin{array}{ll}
p_{\lambda,i}^{(k)} & \text{if       } \lambda,i\in \Gamma(P_k), \text{ for some } k\\
0 & \text{else. } 
\end{array}
\right. 
\]
That is, $P$ is the block matrix
\begin{equation}\label{P form}  P = \begin{bmatrix}
P_1 & 0 & 0 & \cdots  \\
0&P_2 & 0 & \cdots  \\
0 & 0 & P_3  & \ddots  \\
\vdots & \vdots & \ddots & \ddots 
\end{bmatrix}. 
\end{equation} 
We denote $S$ by $\mathlarger{\mathlarger{\circledast}}_{k\in A}^G S_k$. 
The subsemigroups $S_k$ of $S$ are called \textit{Rees components of $S$}. 
Notice that each $\Gamma(P_k)$ is a union of connected components of $\Gamma(P)$. The subsemigroup $S_k$ will be called a \textit{connected Rees component of $S$} if $\Gamma(P_k)$ is connected (and is therefore a connected component of $\Gamma(P)$).

Conversely, for any Rees matrix semigroup $S=\mcal{M}^0[G;I,\Lambda;P]$ there exists partitions $\{I_k:k\in A\}$ and $\{\Lambda_k:k\in A\}$ of $I$ and $\Lambda$, respectively, such that  $\mcal{C}(\Gamma(P))=\{\Lambda_k\cup I_k:k\in A\}$. Consequently, for each $k\in A$, the subsemigroup $S_k=\mcal{M}^0[G;I_k,\Lambda_k;P_k]$ of $S$ is a connected Rees component, where $P_k$ is the $\Lambda_k\times I_k$ submatrix of $P$, and are such that $S_kS_{\ell}=0$ for all $k\neq \ell$. 
Following the work of Graham \cite{Graham68}, we may then permute the rows and columns of $P$ if necessary to assume without loss of generality that $P$ is a block matrix of the form \eqref{P form}. 

Note that if $S$ is a Rees matrix semigroup with connected Rees components $\{S_k:k\in A\}$ then clearly
\begin{equation} \label{E rees} E(S) = \bigcup_{k\in A} E(S_k). 
\end{equation}

Using the fact that automorphisms of $\Gamma(P)$ arise as collections of isomorphisms between its connected components, we obtain an alternative description of automorphisms of a Rees matrix semigroups. The proof is a simple exercise, and can be found in \cite{Quinn}. 

\begin{corollary}\label{iso components} Let $S=\mathlarger{\mathlarger{\circledast}}^G_{k\in A} S_k =\mcal{M}^0[G;I,\Lambda;P]$ be a Rees matrix semigroup such that each $S_k=\mcal{M}^0[G;I_k,\Lambda_k;P_k]$ is a connected Rees component of $S$. 
Let $\pi$ be a bijection of $A$ and, for each $k\in A$, let $\phi_k=(\theta,\psi_k,(u_i^{(k)})_{i\in I_k},(v_{\lambda}^{(k)})_{\lambda\in \Lambda_k})$ be an isomorphism from $S_k$ to $S_{k\pi}$. 
Then $\phi=(\theta,\psi,(u_i)_{i\in I},(v_{\lambda})_{\lambda\in \Lambda})$ is an automorphism of $S$, where $\psi=\bigcup_{k\in A} \psi_k$, and if $i,\lambda\in \Gamma(P_k)$ then $u_i=u_i^{(k)}$ and $v_{\lambda}=v_{\lambda}^{(k)}$. 
Moreover, every automorphism of $S$ can be described in this way. 
\end{corollary} 

We observe that the induced group automorphisms of the isomorphisms $\phi_k$ above must all be equal.

Recall that if $S=\mcal{M}^0[G;I,\Lambda;P]$ is $\aleph_0$-categorical, then $\Gamma(P)$ is $\aleph_0$-categorical by Proposition \ref{G and gamma cat}, and thus $\mcal{C}(\Gamma(P))$ is finite, up to isomorphism, with each connected component being $\aleph_0$-categorical by Proposition \ref{bg iff cc}. 
We extend this result to the set of all connected Rees components of $S$ as follows:   

\begin{proposition}\label{cat crc} Let $S=\mathlarger{\mathlarger{\circledast}}_{k\in A}^G S_k$ be an $\aleph_0$-categorical Rees matrix semigroup such that each $S_k$ is a connected Rees component of $S$. Then each $S_k$ is $\aleph_0$-categorical and $S$ has finitely many connected Rees components, up to isomorphism. 
\end{proposition} 
 
\begin{proof}
We claim that $\{(S_k,a_k):k\in A\}$ is a system of 1-pivoted p.r.c. subsemigroups of $S$ for any $a_k\in S_k^*$, to which the result follows by Proposition \ref{rel-char cat}.
Indeed, let $\phi$ be an automorphism of $S$ such that $a_k \phi =a_l$ for some $k,l$. 
Then, by Corollary \ref{iso components}, there exists a bijection $\pi$ of $A$ with $S_k\phi=S_{k\pi}=S_l$ as required.  
\end{proof} 

Our interest is now in attaining a converse to the proposition above, since it would provide us with a   method for building `new' $\aleph_0$-categorical Rees matrix semigroups from `old'. 
With the aid of Lemma \ref{parts hard sub}, we shall prove that a converse exists in the class of Rees matrix semigroups over finite groups.
 The case where the maximal subgroups are infinite is an open problem. 

Given a pair $S=\mcal{M}^0[G;I,\Lambda;P]$ and $S'=\mcal{M}^0[G;I',\Lambda';Q]$ of Rees matrix semigroups over a group $G$, we denote $\text{Iso}(S;S')(1_{G})$ as the set of isomorphisms between $S$ and $S'$ with trivial induced group isomorphism.
 That is, $\text{Iso}(S;S')(1_{G})$ is the subset of  $\text{Iso}(S;S')$ given by 
\[  \{\phi: \exists \psi\in \text{Iso}(\Gamma(P);\Gamma(Q)) \text{ and }  u_i, v_{\lambda}\in G \text{ such that }  \phi = (1_G,\psi,(u_i)_{i\in I},(v_{\lambda})_{\lambda\in \Lambda})\}. 
\] 
If $S=S'$ we denote this simply as \aut$(S)(1_G)$, and notice that Aut$(S)(1_G)$ is a subgroup of Aut$(S)$ by Corollary \ref{induced iso}. 

\begin{lemma}\label{cat iff I equiv} Let $S=\mcal{M}^0[G;I,\Lambda;P]$ be a Rees matrix semigroup over a finite group $G$. 
Then $S$ is $\aleph_0$-categorical if and only if $S$ is $\aleph_0$-categorical over Aut$(S)(1_G)$. 
\end{lemma} 

\begin{proof}
Let $S$ be $\aleph_0$-categorical with $G=\{g_1,\dots,g_r\}$ finite. Let $\overline{a}$ and $\overline{b}$ be a pair of $n$-tuples of $S$.
 For some fixed $p_{\mu,j}\neq 0$, let $\overline{g}$ be the $r$-tuple of $S$ given by 
$\overline{g}=((j,g_1,\mu),\dots,(j,g_r,\mu))$, and suppose that $(\overline{a},\overline{g})\, \sim_{S,n+r} \, (\overline{b},\overline{g})$ via $\phi=(\theta,\psi,(u_i)_{i\in I},(v_{\lambda})_{\lambda\in \Lambda})$, say.
 Then, for each $1\leq k\leq r$, we have
\[ (j,g_k,\mu)\phi=(j\psi,u_j(g_k\theta)v_{\mu},\mu\psi)=(j,g_k,\mu), 
\] 
 so that $g_k\theta=u_j^{-1}g_k v_{\mu}^{-1}$. For each $i\in I, \lambda \in \Lambda$, let $\bar{u}_i=u_iu_j^{-1}$ and $\bar{v}_{\lambda}=v_{\mu}^{-1}v_{\lambda}$.
  Then
\begin{align*} (i\psi, \bar{u}_i g_k \bar{v}_\lambda,\lambda\psi) & =(i\psi, (u_iu_j^{-1}) g_k (v_{\mu}^{-1}v_\lambda),\lambda\psi) \\
& =(i\psi,u_i(g_k\theta)v_{\lambda},\lambda\psi) \\
& =(i,g_k,\lambda)\phi, 
\end{align*} 
for any $(i,g_k,\lambda)\in S$, so that $\phi=(1_G,\psi,(\bar{u}_i)_{i\in I},(\bar{v}_\lambda)_{\lambda\in \Lambda})\in \text{Aut}(S)(1_G)$. Consequently, $(\overline{a},\overline{g})\, \sim_{S,\text{Aut}(S)(1_G),n+r} \, (\overline{b},\overline{g})$ 
and in particular $\overline{a}\, \sim_{S,\text{Aut}(S)(1_G),n} \, \overline{b}$. We have thus shown that 
\[ |S^n/\sim_{S,\text{Aut}(S)(1_G),n}|\leq |S^{n+r}/\sim_{S,n+r}|<\aleph_0,
\] 
as $S$ is $\aleph_0$-categorical. Hence $S$ is $\aleph_0$-categorical over Aut$(S)(1_G)$. 

The converse is immediate.  
\end{proof}

We are now able to prove our desired converse to Proposition \ref{cat crc} in the case where the maximal subgroups are finite. 

\begin{theorem}\label{thm G finite} Let $S=\mcal{M}^0[G;I,\Lambda;P]$ be a Rees matrix semigroup such that $G$ is finite. 
Then $S$ is $\aleph_0$-categorical if and only if each connected Rees component of $S$ is $\aleph_0$-categorical and $S$ has only finitely many connected Rees components, up to isomorphism. 
\end{theorem} 
 
\begin{proof}  ($\Rightarrow$) Immediate from Proposition \ref{cat crc}. 

($\Leftarrow$) Since $S$ is regular with finite maximal subgroups, to prove $S$ is $\aleph_0$-categorical, it suffices by \cite[Corollary 3.14]{Quinncat} to show that $|E(S)^n/\sim_{S,n}|$ is finite, for each $n\in \mathbb{N}$. 
Let $\{S_k:k\in A\}$ be the set of connected Rees components of $S$, which is finite up to isomorphism and with each $S_k$ being $\aleph_0$-categorical.
 Define a relation $\eta$ on $A$ by $i \, \eta \, j$ if and only if $\text{\iso}(S_i;S_j)(1_G)\neq \emptyset$. By Corollary \ref{induced iso} we have that $\eta$ is an equivalence relation. 

We first prove that $A/\eta$ is finite. Suppose for contradiction that there exists an infinite set $X$ of pairwise $\eta$-inequivalent elements of $A$. 
Since $S$ has finitely many connected components up to isomorphism, there exists an infinite subset $\{i_r:r\in \mathbb{N}\}$ of $X$ such that $S_{i_n}\cong S_{i_m}$ for each $n,m$. Fix an isomorphism $\phi_{i_n}:S_{i_n}\rightarrow S_{i_1}$ for each $n\in\mathbb{N}$. 
Then as Aut$(G)$ is finite there exists distinct $n,m$ such that $\phi_{i_n}^G=\phi_{i_m}^G$, and so $\phi_{i_n}\phi_{i_m}^{-1}\in \text{\iso}(S_{i_n};S_{j_m})(1_G)$ by Corollary \ref{induced iso}. 
Hence $i_n \, \eta \, i_m$, a contradiction, and so $A/\eta$ is finite. 

Let $S'=\bigcup_{k\in A} S_k$, noting that $S'$ is the 0-direct union of the $S_k$, and in particular is a subsemigroup of $S$. 
 Let $A/\eta = \{A_1,\dots,A_r\}$ and set $\underline{A}=(A;A_1,\dots,A_r)$.
For each $i,j\in A$, let $\Psi_{i,j}= \text{\iso}(S_i;S_j)(1_G)$ and fix $\Psi=\bigcup_{i,j\in{A}}\Psi_{i,j}$.  We   prove that $\{S_k:k\in A\}$ forms an $(S;S'; \underline{A};\Psi)$-system in $S$. 
First, by our construction, if $i,j\in A_m$ for some $m$ then  $\Psi_{i,j} \neq \emptyset$, and so $\Psi$ satisfies Condition (3.1). 
Furthermore, it follows immediately from Corollary \ref{induced iso} that $\Psi$ satisfies Conditions (3.2) and (3.3). 
Finally, take any $\pi\in \text{Aut}(\underline{A})$ and, for each $k\in A$, let $\phi_k\in \Psi_{k,k\pi}$. 
Then as $\phi_k^G=1_G$ for each $k\in A$, we may construct an automorphism $\phi$ of $S$ from the set of isomorphisms $\{\phi_k:k\in A\}$ by Corollary \ref{iso components}. 
Hence, as $\phi$ extends each $\phi_k$ by construction, we have that $\{S_k:k\in A\}$ forms an $(S;S';\underline{A};\Psi)$-system as required. 
Since $S_k$ is $\aleph_0$-categorical, it is $\aleph_0$-categorical over $\Psi_{k,k}=\text{Aut}(S_k)(1_G)$ by Lemma \ref{cat iff I equiv}. By Corollary \ref{set plus cat} $\underline{A}$ is $\aleph_0$-categorical, and so 
\[ |(S')^n/\sim_{S,n}|<\aleph_0
\] 
by Lemma \ref{parts hard sub}. 
Given that $E(S)  \subseteq S'$ by \eqref{E rees}, we therefore have that
\[ |E(S)^n/\sim_{S,n}|\leq |(S')^n/\sim_{S,n}|<\aleph_0.
\]
 Hence $S$ is $\aleph_0$-categorical.   
 \end{proof}

\begin{open} Does Theorem \ref{thm G finite} hold if $G$ is allowed to be any $\aleph_0$-categorical group?
\end{open}

\subsection{Labelled bipartite graphs} 

In Example \ref{example main}, the problem which arose was that by shifting from the sandwich matrix $P=(p_{\lambda,i})$ to the induced bipartite graph $\Gamma(P)$ we have ``forgotten'' the value of the entries $p_{\lambda,i}$. 
In this subsection we extend the construction of the induced bipartite graph of a Rees matrix semigroup to attempt to rectifying this problem, as well as to build classes of $\aleph_0$-categorical Rees matrix semigroups. 
Further examples of $\aleph_0$-categorical Rees matrix semigroups can then be built using  Theorem \ref{thm G finite}. 

\begin{definition} Let $\Gamma=\langle L,R,E \rangle$ be a bipartite graph, $\Sigma$ a set, and $f\colon E\rightarrow \Sigma$ a surjective map. Then the triple $(\Gamma,\Sigma,f)$ is called a \textit{$\Sigma$-labeled (by $f$) bipartite graph}, which we denote as $\Gamma^f$.
\end{definition} 

A pair  of $\Sigma$-labeled bipartite graphs $\Gamma^f=(\Gamma, \Sigma,f)$ and $\Gamma^{f'}=(\Gamma', \Sigma,f')$ are \textit{isomorphic} if there exists an isomorphism $\psi\colon \Gamma\rightarrow \Gamma'$ which preserves labels, that is, such that  
\[ \{x,y\} f=\sigma \Leftrightarrow \{x\psi,y\psi\}f'=\sigma. 
\] 
This gives rise to a natural signature in which to consider $\Sigma$-labeled bipartite graphs as follows. 
 For each $\sigma\in \Sigma$, take a binary relation symbol $E_\sigma$ and let 
\[ L_{BG\Sigma}=L_{BG} \cup \{E_{\sigma}:\sigma\in \Sigma\}. 
\] 
Then we call $L_{BG\Sigma}$ the \textit{signature of $\Sigma$-labeled bipartite graphs}, where $(x,y)\in E_\sigma$ if and only if $\{x,y\}\in E$ and $\{x,y\}f=\sigma$.

Let $\Gamma^f$ be a $\Sigma$-labeled bipartite graph. Then for any set $\Sigma'$ and bijection $g\colon \Sigma\rightarrow \Sigma'$,  we can form a $\Sigma'$-labeling of $\Gamma$ simply by taking $\Gamma^{fg}$, which we call a \textit{relabeling of $\Gamma^f$}. Notice that if $\psi$ is an automorphism of $\Gamma$, then $\psi\in \text{Aut}(\Gamma^f)$ if and only if $\psi\in \text{Aut}(\Gamma^{fg})$. Indeed, if  $\psi\in \text{Aut}(\Gamma^f)$ then for any edge $\{x,y\}$ of $\Gamma$ we have 
\[ \{x,y\}fg=\sigma' \Leftrightarrow \{x,y\}f=\sigma'g^{-1} \Leftrightarrow \{x\psi,y\psi\}f=\sigma'g^{-1} \Leftrightarrow \{x\psi,y\psi\}fg = \sigma',
\] 
since $g$ is a bijection. The converse is proven similarly, and the following result is then immediate from the RNT.  

\begin{lemma} Let $\Gamma^f$ be a $\Sigma$-labeling of a bipartite graph $\Gamma$. Then $\Gamma^f$ is $\aleph_0$-categorical if and only if any relabeling of $\Gamma^f$ is $\aleph_0$-categorical. 
\end{lemma} 

\begin{lemma}\label{label finite} If $\Gamma^f=(\Gamma, \Sigma,f)$ is an $\aleph_0$-categorical labeled bipartite graph then $\Sigma$ is finite and $\Gamma$ is $\aleph_0$-categorical.  
\end{lemma}

\begin{proof} For each $\sigma \in \Sigma$, let $\{x_{\sigma},y_{\sigma}\}$ be an edge in $\Gamma$  such that $\{x_{\sigma},y_\sigma\}f=\sigma$. Then $\{(x_{\sigma},y_\sigma):\sigma\in \Sigma\}$ is a set of distinct 2-automorphism types of $\Gamma^f$, and so $\Sigma$ is finite by the RNT. 
Since automorphisms of $\Gamma^f$ induce automorphisms of $\Gamma$,  the final result is immediate from the RNT.  
\end{proof}  

A consequence of the previous pair of lemmas is that, in the context of $\aleph_0$-categoricity,  it suffices to consider finitely labeled bipartite graphs, with labeling set ${\bf m}=\{1,2,\dots,m\}$ for some $m\in \mathbb{N}$.

\begin{lemma}\label{finite L} Let $\Gamma^f=(\langle L,R,E \rangle, {\bf m},f)$ be an ${\bf m}$-labeled bipartite graph such that either $L$ or $R$ are finite. Then $\Gamma^f$ is $\aleph_0$-categorical. 
\end{lemma}  

\begin{proof}
Without loss of generality assume that $L=\{l_1,l_2,\dots, l_r\}$ is finite. Define a relation $\tau$ on $R$ by $y \, \tau \, y'$ if and only if $y$ and $y'$ are adjacent to the same elements in $L$ and $\{l_i,y\}f = \{l_i,y'\}f$ for each such $l_i\in L$. Note that since both $L$ and {\bf m} are finite, $R$ has  finitely many $\tau$-classes, say $R_1,\dots,R_t$. 
Considering $R$ simply as a set, fix $\mcal{A}=(R;R_1,\dots,R_t)$. 

Since $L$ is finite, to prove the $\aleph_0$-categoricity of $\Gamma^f$   it suffices to show that $(\Gamma^f\setminus L)^n= R^n$ has finitely many $\sim_{\Gamma^f,n}$-classes for each $n\in \mathbb{N}$ by a simple generalization of \cite[Proposition 2.11]{Quinncat}. 
Let $\overline{a}=(r_1,\dots,r_n)$ and $\overline{b}=(r_1',\dots,r_n')$ be $n$-tuples of $R$ such that $\overline{a} \, \sim_{\mcal{A},n} \, \overline{b}$ via $\psi\in \text{Aut}(\mcal{A})$, say.
 We claim that the map $\hat{\psi}\colon \Gamma^f\rightarrow \Gamma^f$ which fixes $L$ and is such that $\hat{\psi}|_R=\psi$ is an automorphism of $\Gamma^f$.
  Indeed, as $\psi$ setwise fixes the $\tau$-classes, we have $(r,r\psi)\in \tau$ for each $r\in R$. Hence $r$ and $r\psi$ are adjacent to the same elements in $L$, and so
\[ \{l_i,r\}\in E \Leftrightarrow \{l_i,r\psi\} \in E \Leftrightarrow \{l_i\hat{\psi},r\hat{\psi}\}\in E,
\]
so that $\hat{\psi}$ is an automorphism of $\Gamma$. Similarly $\{l_i,r\}f=\{l_i,r\psi\}f=\{l_i\hat{\psi},r\hat{\psi}\}f$, so that $\hat{\psi}$ preserves labels. This proves the claim. 

For each $1\leq k \leq n$ we have $r_k\hat{\psi}= r_k\psi=r_k'$, so that $\overline{a} \, \sim_{\Gamma^f,n} \, \overline{b}$. Consequently,  
\[ |(\Gamma^f\setminus L)^{n}/\sim_{\Gamma^f,n}| \leq |\mcal{A}^n/\sim_{\mcal{A},n}|.
\] 
The set extension $\mcal{A}$ is $\aleph_0$-categorical by  Corollary \ref{set plus cat}, and so $|\mcal{A}^n/\sim_{\mcal{A},n}|$ is finite for each $n\geq 1$. Hence $\Gamma^f$ is $\aleph_0$-categorical.  
\end{proof}

\begin{lemma}\label{finite label} Let $\Gamma^f=(\langle L,R,E \rangle, {\bf m},f)$ be such that there exists $p\in {\bf m}$ with $\{x,y\}f=p$ for all but finitely many edges in $\Gamma$. Then $\Gamma^f$ is $\aleph_0$-categorical if and only if $\Gamma$ is $\aleph_0$-categorical. 
\end{lemma} 

\begin{proof} 
 Suppose $\Gamma$ is $\aleph_0$-categorical, and that $\{l_1,r_1\},\dots, \{l_t,r_t\}$ are precisely the edges of $\Gamma$ such that $\{l_k,r_k\}f\neq p$, where $l_k\in L$ and $r_k\in R$. Let $\overline{a}$ and $\overline{b}$ be $n$-tuples of $\Gamma^f$  such that 
\[ (\overline{a},l_1,r_1,\dots,l_t,r_t) \, \sim_{\Gamma,n+2t} \, (\overline{b},l_1,r_1,\dots,l_t,r_t)
\] 
via $ \psi\in \text{Aut}(\Gamma)$, say. 
We claim that $\psi$ is an automorphism of $\Gamma^f$. For each $1\leq k \leq t$ we have  $l_k\psi=l_k$ and $r_k\psi=r_k$ so that
\[\{l_k,r_k\}f= \{l_k\psi,r_k\psi\}f. 
\]
It follows that $\{l,r\}f=p$ if and only if $\{l\psi,r\psi\}f=p$, and so $\psi$ preserves all labels, thus proving the claim. Consequently, $\overline{a} \, \sim_{\Gamma^f,n} \, \overline{b}$ via $\psi$, so that
\[ |(\Gamma^f)^{n} /\sim_{\Gamma^f,n}| \leq |\Gamma^{n+2t}/\sim_{\Gamma,n+2t}|<\aleph_0
\] 
by the $\aleph_0$-categoricity of $\Gamma$. Hence $\Gamma^f$ is $\aleph_0$-categorical. 

The converse is immediate from Lemma \ref{label finite}.  
\end{proof}

\begin{definition} Given a Rees matrix semigroup $S=\mcal{M}^0[G;I,\Lambda;P]$, we form a $G(P)$-labeling of the induced bipartite graph  $\Gamma(P) = \langle I,\Lambda, E\rangle$ of $S$ in the natural way by taking the labeling $f\colon E\rightarrow G(P)$ given by 
\[ \{i,\lambda\}f = p_{\lambda,i}.
\] 
We denote the labeled bipartite graph by $\Gamma(P)^l$, which we call the \textit{induced labeled bipartite graph} of $S$.
\end{definition} 

 Note that, unlike the corresponding case for the induced bipartite graph $\Gamma(P)$,  there exist isomorphic Rees matrix semigroups with non-isomorphic induced labeled bipartite graphs. 
 For example, let $G$ be a non-trivial group and $P$ and $Q$  be $\mathbf{1} \times \mathbf{2}$ matrices over $G\cup \{0\}$ given by 
\[ P =  \left( \begin{array}{cc}
1 & a 
\end{array} \right)
\quad Q = \left( \begin{array}{cc}
1 & 1 
\end{array} \right)
\]
where $a\notin \{0,1\}$. Let $S=\mcal{M}^0[G;\mathbf{2},\mathbf{1};P]$ and  $T=\mcal{M}^0[G;\mathbf{2},\mathbf{1};Q]$, noting that $\Gamma(P)=\Gamma(Q)$ (and are isomorphic to $K_{2,1}$).
 Then $(1_G,1_{\Gamma(P)},(u_i)_{i\in \mathbf{2}},(v_{\lambda})_{\lambda\in \mathbf{1}})$ is an isomorphism from $S$ to $T$, where $u_1=1=v_1$, and $u_2=a$. 
 However, since $\Gamma(P)^l$ and $\Gamma(Q)^l$ have different labeling sets, they are not isomorphic. 

\begin{proposition}\label{label cat rees} Let $S=\mcal{M}^0[G;I,\Lambda;P]$ be a Rees matrix semigroup such that $G$ and $\Gamma(P)^l$ are $\aleph_0$-categorical. 
Then $S$ is $\aleph_0$-categorical. 
\end{proposition}  

\begin{proof} Since $\Gamma(P)^l$ is $\aleph_0$-categorical, the set $G(P)$ is finite by Lemma \ref{label finite}, say $G(P)=\{x_1,\dots,x_r\}$. Consider a pair of $n$-tuples $\overline{a}=((i_1,g_1,\lambda_1),\dots, (i_n,g_n,\lambda_n))$ and $\overline{b}=((j_1,h_1,\mu_1),\dots, (j_n,h_n,\mu_n))$   of $S^*$ under the pair of conditions that 
\begin{itemize}
\item[(1)] $(g_1,\dots,g_n,x_1,\dots,x_r) \, \sim_{G,n+r} \, (h_1,\dots ,h_n,x_1,\dots,x_r)$,
\item[(2)] $\Gamma(\overline{a}) \, \sim_{\Gamma(P)^l,2n} \, \Gamma(\overline{b})$,
\end{itemize}
via $\theta\in \text{Aut}(G)$ and  $\psi\in \text{Aut}(\Gamma(P)^l)$, respectively (noting the use of Notation \ref{tuple not} here).  
We claim that $\phi=(\theta,\psi,(1)_{i\in I},(1)_{\lambda\in \Lambda})$ is an automorphism of $S$.
 Indeed, if $p_{\lambda,i}\neq 0$ for some $i\in I,\lambda\in \Lambda$, then  $p_{\lambda,i}=x_k$ for some $k$, so that $\{i,\lambda\}f=\{i\psi,\lambda\psi\}f=x_k$.
 Consequently,  
\[  p_{\lambda,i}\theta=x_k\theta=x_k = p_{\lambda\psi,i\psi}, 
\] 
and the claim  follows by Theorem \ref{iso c0s}.
 Hence 
\[ (i_t,g_t,\lambda_t)\phi=(i_t\psi,g_t\theta,\lambda_t\psi)=(j_t,h_t,\mu_t)
\] 
 for each $1\leq t \leq n$, so that
\[ |(S^*)^n/\sim_{S,n}|\leq |G^{n+r}/\sim_{G,n+r}|\cdot |(\Gamma(P)^l)^{2n}/\sim_{\Gamma(P)^l,2n}|<\aleph_0,
\] 
as $G$ and $\Gamma(P)^l$ are $\aleph_0$-categorical.
 Hence $S$ is $\aleph_0$-categorical by \cite[Proposition 2.11]{Quinncat}.  
\end{proof}

 The proposition above enables us to produce concrete   examples of $\aleph_0$-categorical Rees matrix semigroups. 
For example, the result below is immediate from Lemma \ref{finite L}. 

\begin{corollary} Let $S$ be a Rees matrix semigroup over an $\aleph_0$-categorical group having sandwich matrix $P$ with finitely many rows or columns, and $G(P)$ being finite. 
Then $S$ is $\aleph_0$-categorical.
\end{corollary} 

Similarly, Lemma \ref{finite label} may be used in conjunction with Proposition \ref{label cat rees} to obtain: 

\begin{corollary}\label{finite entries} Let $S=\mcal{M}^0[G;I,\Lambda;P]$ be a Rees matrix semigroup such that $G$ and $\Gamma(P)$ are $\aleph_0$-categorical, and all but finitely many of the non-zero entries of $P$ are the identity of $G$. Then $S$ is $\aleph_0$-categorical.
\end{corollary}

However, the converse to Proposition \ref{label cat rees}    fails to hold in general, and a counterexample will be constructed  later in the next subsection. 
The idea is that any $\mcal{M}^0[G;I,\Lambda;P]$ in which $G(P)$ is infinite forces $\Gamma(P)^l$ to be non $\aleph_0$-categorical by Lemma \ref{label cat rees}. 

\begin{open}\label{open 1}  Does there exist an $\aleph_0$-categorical connected Rees matrix semigroup with $G(P)$ finite which is not isomorphic to a Rees matrix semigroup with $\aleph_0$-categorical induced labeled bipartite graph? 
\end{open}

We prove that the open problem has a negative answer for the case of completely simple semigroups. Given a completely simple semigroup $\mcal{M}[G;I,\Lambda;P]$, we call $P$ \textit{normal} if there exist $i\in I$ and $\lambda\in \Lambda$ such that $p_{\mu,i}=p_{\lambda,j}=1$ for every $j\in I$ and $\mu\in \Lambda$. 
Every completely simple semigroup is isomorphic to a Rees matrix semigroup without zero in which the sandwich matrix is normal \cite{Howie94}. 

\begin{proposition}\label{prop:css} Let $\mcal{M}[G;I,\Lambda;P]$ be an $\aleph_0$-categorical completely simple semigroup in which  $P$ is normal and $G(P)$ is finite. 
Then $\Gamma(P)^l$ is $\aleph_0$-categorical. 
\end{proposition} 

\begin{proof} Suppose $P$ is normalised via $i^*\in I$ and $\lambda^*\in\Lambda$.   Since $G(P)$ is finite we may fix some finite subsets $I'=\{x_1,\dots,x_p\}\subseteq I$ and $\Lambda'=\{y_1,\dots,y_q\}\subseteq \Lambda$ such that the $\Lambda'\times I'$ submatrix of $P$ contains every element of $G(P)$. Let $\overline{x}$ be the $pq$-tuple of $S$ given by  
\[  ((x_1,1,y_1),(x_1,1,y_2),\dots,(x_1,1,y_q), (x_2,1,y_1),\dots,(x_p,1,y_q)),
\]  Using the notation of Proposition \ref{G and gamma cat}, let $\overline{a}=(a_1,\dots,a_n)$ and $\overline{b}=(b_1,\dots,b_n)$ be a pair of $\sigma_{\Gamma(P)^l,n}$-related $n$-tuples of $\Gamma(P)^l$. Let $i_1<i_2<\cdots <i_s$ and $j_1<j_2<\cdots<j_t$ be the indexes of entries of $\overline{a}$ (and thus $\overline{b}$) lying in $I$ and $\Lambda$, respectively. Suppose further that there exists $i\in I$ and $\lambda\in \Lambda$ such that the $n+pq+1$-tuples
\begin{align*}
& ((a_{i_1},1,\lambda),\dots,(a_{i_s},1,\lambda),(i,1,a_{j_1}),\dots,(i,1,a_{j_t}), \overline{x},(i^*,1,\lambda^*)) \text{ and }  \\
& ((b_{i_1},1,\lambda),\dots,(b_{i_s},1,\lambda),(i,1,b_{j_1}),\dots,(i,1,b_{j_t}),\overline{x},(i^*,1,\lambda^*)) 
\end{align*}
are automorphically equivalent via $\phi=[\theta,\psi,(u_i)_{i\in I},(v_{\lambda})_{\lambda\in \Lambda}] \in \text{Aut}(S)$, say. Then $\psi$ is an automorphism of $\Gamma(P)$ which maps $\overline{a}$ to $\overline{b}$. We aim to show that $\psi$ preserves labels, i.e., $p_{\lambda,i}=p_{\lambda\psi,i\psi}$ for every $i\in I, \lambda\in \Lambda$. Since $\phi$ fixes $(i^*,1,\lambda^*)$ we have by \cite[Corollary 4.9]{Quinncss}  that there exists $g\in G$ with $u_i=g$ and $v_{\lambda}=g^{-1}$ for every $i\in I, \lambda\in \Lambda$. Since $\psi$ fixes $x_1,\dots,x_p,y_1,\dots,y_q$ we have $$p_{y_k,x_{\ell}}\theta = v_{y_k} p_{y_k\psi, x_{\ell}\psi} u_{x_{\ell}} = g^{-1} p_{y_k,x_{\ell}} g.$$
Consequently, as every $p_{\lambda,i}$ is equal to some $p_{y_k,x_{\ell}}$, we have $p_{\lambda,i}\theta = g^{-1} p_{\lambda,i} g$ for every $i\in I$, $\lambda\in \Lambda$. However, $p_{\lambda,i}\theta=v_\lambda p_{\lambda\psi,i\psi}u_i=g^{-1} p_{\lambda\psi,i\psi} g$, and hence $\psi$ preserves labels as required. 
We have thus shown that 
\[ (\Gamma(P)^l)^n / \sim_{\Gamma(P)^l,n}| \leq |S^{n+pq+1}/\sim_{S, n+pq+1}| <\aleph_0
\]
as $S$ is $\aleph_0$-categorical.  
\end{proof}

The sandwich matrix $P$ of a Rees matrix semigroup $\mcal{M}^0[G;I,\Lambda;P]$ can also always be normalised, but it is necessarily more complex. We can restate Open Problem \ref{open 1} as follows:

\begin{open}  If $\mcal{M}^0[G;I,\Lambda;P]$ is $\aleph_0$-categorical, where   $P$ is normal and $G(P)$ is finite, then is $\Gamma(P)^l$   $\aleph_0$-categorical?
\end{open} 
 
Notice that in Example \ref{example main}, the labeled bipartite graph is clearly not $\aleph_0$-categorical since each $i_k$ is adjacent to exactly $k$ vertices in which the edge is labeled by $a$. By construction the matrix $P$ is normal via row $\alpha_0$ and column $i_0$, and hence $S$ is not $\aleph_0$-categorical by the proposition above.

 \subsection{Pure completely 0-semigroups} 
 
 Following \cite{Jackson}, we call a completely 0-simple semigroup $S$ \textit{pure} if it is isomorphic to a Rees matrix semigroup with sandwich matrix over $\{0,1\}$. 
  In \cite{Houghton77}, Houghton considered \textit{trivial cohomology classes} of Rees matrix semigroups, a property which is proven in Section 2 of his article to be equivalent to being pure. Hence, by \cite[Theorem 5.1]{Houghton77}, a completely 0-simple semigroup is pure if and only if, for each $a,b\in S$, 
\[ [a,b\in \langle E(S) \rangle \text{ and } a \, \mcal{H} \, b] \Rightarrow a=b.
\]
It follows that all orthodox completely 0-simple semigroups are necessarily pure, but the converse is not true in general. Indeed, a completely 0-simple semigroup is orthodox if and only if it is isomorphic to a Rees matrix semigroup with sandwich matrix over $\{0,1\}$ and with induced bipartite graph a disjoint union of complete bipartite graphs \cite[Theorem 6]{Hall}. Hence, in this case, it can be easily shown that the isomorphism types of the connected Rees components depends only on the isomorphism types of the induced (complete) bipartite graphs. 

We observe that if the sandwich matrix of a Rees matrix semigroup is over $\{0,1\}$ then $\Gamma(P)^l$ is simply labeled by $\{1\}$. 
Therefore all automorphisms of $\Gamma(P)$ automatically preserve the labeling, and so $\Gamma(P)^l$ is $\aleph_0$-categorical if and only if $\Gamma(P)$ is $\aleph_0$-categorical. 
The equivalence of statements (1),(3), and (4) in the result below therefore follow from Propositions \ref{G and gamma cat} and \ref{label cat rees}. For the interest of the reader we give an alternative proof of (4) $\Rightarrow$ (1) using results in \cite{Quinncat}.  


\begin{lemma}\label{lemm:new} Let $S=\mcal{M}^0[G;I,\Lambda;P]$ be a  pure Rees matrix semigroup. Then the following are equivalent: 
\begin{enumerate}
\item[(1)] $S$ is $\aleph_0$-categorical; 
\item[(2)]  $G$ and $\langle E(S)\rangle $ are $\aleph_0$-categorical; 
\item[(3)] $G$ and $\Gamma(P)$ are $\aleph_0$-categorical; 
\item[(4)] $G$ and $\mcal{M}^0[\{1\};I,\Lambda;P]$ are $\aleph_0$-categorical. 

\end{enumerate}  
\end{lemma} 

\begin{proof}
 (1) $\Rightarrow$ (2) If $S$ is $\aleph_0$-categorical then so is $G$ by Proposition \ref{G and gamma cat}. Clearly $E(S)$ is preserved by  automorphisms of $S$, and hence $\langle E(S) \rangle$ is a characteristic subsemigroup of $S$, and thus inherits $\aleph_0$-categoricity. 
 
(2) $\Rightarrow$ (3)  Suppose that $\langle E(S) \rangle=\langle \{(i,1,\lambda):p_{\lambda,i}\neq 0\}\cup \{0\}\rangle$ is $\aleph_0$-categorical.  
Let $S_k=\mcal{M}^0[G;I_k,\Lambda_k;P_k]$ ($k\in A$) be the connected Rees components of $S$, where  $P_k$ is the $\Lambda_k\times I_k$ submatrix of $P$.
Then  $\langle E(S) \rangle$ is isomorphic to the 0-direct union of the semigroups $E_k=\langle E(S_k) \rangle$, and since each $P_k$ is regular it is a simple exercise to show that $E_k = \mcal{M}^0[\{1\};I_k,\Lambda_k;P_k]$. By \cite[Corollary 4.9]{Quinncat} $\langle E(S) \rangle$ is $\aleph_0$-categorical if and only if each $E_k$  is $\aleph_0$-categorical and $\{E_k: k\in A\}$ is finite, up to isomorphism. By Proposition \ref{G and gamma cat} each $\Gamma(P_k)$ is $\aleph_0$-categorical, and by Theorem~\ref{iso c0s} $\mcal{C}(\Gamma(P))=\{\Gamma(P_k):k\in A\}$ is finite, up to isomorphism. Hence $\Gamma(P)$ is $\aleph_0$-categorical by Proposition \ref{bg iff cc}. 
 
 (3) $\Rightarrow$ (4) Immediate from Corollary \ref{finite entries}. 
 
 (4) $\Rightarrow$ (1) The elements of the combinatorial Rees matrix semigroup $T=\mcal{M}^0[\{1\};I,\Lambda;P]$  can be identified\footnote{Semigroups of this form are known as \textit{rectangular 0-bands.}} with the set $(I\times \Lambda)\cup \{0\}$. Since $\aleph_0$-categoricity is preserved by finite direct products \cite{Grzeg}, the semigroup $U=G \times T$ is $\aleph_0$-categorical. 
 The set $I=\{(g,0):g\in G\}$ is an ideal of $U$, and the Rees quotient $U/I$ is a principal factor of $U$. Hence $U/I$ is $\aleph_0$-categorical by \cite[Theorem 3.12]{Quinncat}. 
 Moreover, the map $\phi\colon U/I\rightarrow S$ given by $0\phi=0$ and $(g,(i,\lambda))\phi = (i,g,\lambda)$ ($g\in G, i\in I,\lambda\in \Lambda$) is an isomorphism, to which the result follows.   
\end{proof}

Furthermore, since complete bipartite graphs are $\aleph_0$-categorical by Theorem \ref{BG CAT LIST}, a disjoint union of complete bipartite graphs is $\aleph_0$-categorical if and only if it has finitely many connected components, up to isomorphism, by Proposition \ref{bg iff cc}. 
The corollary above thus reduces in the orthodox case as follows.   

\begin{corollary}\label{orthod rees} Let $S=\mcal{M}^0[G;I,\Lambda;P]$ be an orthodox Rees matrix semigroup. Then the following are equivalent: 
\begin{enumerate}
\item[(1)] $S$ is $\aleph_0$-categorical;
\item[(2)] $G$ and $E(S)$ are $\aleph_0$-categorical; 
\item[(3)] $G$ is  $\aleph_0$-categorical  and $\Gamma(P)$ has finitely many connected components, up to isomorphism; 
\item[(4)] $G$ and $\mcal{M}^0[\{1\};I,\Lambda;P]$ are $\aleph_0$-categorical. 
\end{enumerate}  
\end{corollary} 


In \cite{Quinncat} we studied inverse completely 0-simple semigroups, that is, Brandt semigroups. These are necessarily orthodox, and are isomorphic to a Rees matrix semigroup of the form $\mcal{M}^0[G;I,I;P]$ where $P$ is the identity matrix, that is, $p_{ii}=1$ and $p_{ij}=0$ for each $i\neq j$ in $I$, and are denoted $\mcal{B}^0[G;I]$. 
  Since the induced biparite graph of a Brandt semigroup is a perfect matching, it is $\aleph_0$-categorical by Theorem \ref{BG CAT LIST}. 
  Corollary \ref{orthod rees} then simplifies to obtain our classification of $\aleph_0$-categorical Brandt semigroups \cite[Theorem 4.2]{Quinncat}, which states that a Brandt semigroup over a group $G$ is $\aleph_0$-categorical if and only if $G$ is $\aleph_0$-categorical. 

We are now able to construct a simple counterexample to the converse of Proposition \ref{label cat rees}. Let $G=\{g_i:i\in \mathbb{N}\}$ be an infinite $\aleph_0$-categorical group.
 Let 
\[ S=\mcal{M}^0[G;\mathbb{N},\mathbb{N};P]=\mcal{B}^0[G;\mathbb{N}] \text{ and } T=\mcal{M}^0[G;\mathbb{N},\mathbb{N};Q],
\]  where $Q=(q_{i,j})$ is such that $q_{i,i}=g_i$ and $q_{i,j}=0$ for each $i\neq j$. 
Then $\Gamma(P)=\Gamma(Q)$ (and are isomorphic to $P_{\mathbb{N}}$) and  $(1_G,1_{\Gamma(P)},(g^{-1}_i)_{i\in \mathbb{N}},(1)_{\lambda \in \mathbb{N}})$ is an isomorphism from $S$ to $T$ by Theorem \ref{iso c0s} since 
\[ p_{i,i}1_G=1=g_ig^{-1}_i =  1 \cdot q_{i,i} \cdot g^{-1}_i, 
\] 
for each $i\in \mathbb{N}$. Since $S$ is $\aleph_0$-categorical by the $\aleph_0$-categoricity of $G$, the same is true of $T$.
 However, $\Gamma(Q)^l$ is a $G$-labeling, and is thus not $\aleph_0$-categorical by Lemma \ref{label finite}. Hence  $T$ is our desired counterexample. 
 
 \subsection{Alternative directions} 
 
To further incorporate the link between the induced bipartite graph of a Rees matrix semigroup and the entries of the sandwich matrix, we  could instead introduce the stronger notion of an \textit{induced group labeled bipartite graph}. 
A group labeled bipartite graph is a $G$-labeled bipartite graph $\Gamma^f=(\langle L,R,E \rangle,G,f)$, for some group $G$, where an automorphism of $\Gamma^f$ is a pair $(\psi,\theta)\in \text{Aut}(\Gamma)\times  \text{Aut}(G)$  such that, for each $\ell \in L, r\in R$, 
\[ (\ell,r)f=g \Leftrightarrow (\ell\psi,r\psi)f=g\theta.  
\]  
However, group labeled biparite graphs do not appear to be first-order structures. 

Let $S=\mcal{M}^0[G;I,\Lambda;P]$ be such that $G(P)$ forms a subgroup of $G$. 
Then we may define the \textit{induced group labeled bipartite graph} of $S$ as the $G(P)$-labeled bipartite graph $\Gamma(P)^f$, with automorphisms being pairs $(\psi,\theta)\in \text{Aut}(\Gamma) \times \text{Aut}(G(P))$ such that  $p_{\lambda\psi,i\psi}=p_{\lambda,i}\theta$ for each $i\in I, \lambda\in \Lambda$. 
Notice that if $(\psi,\theta)$ is an automorphism of the induced group labeled bipartite graph of $S$ and is such that $\theta$ extends to an automorphism $\theta'$ of $G$, then $(\theta',\psi,(1)_{i\in I},(1)_{\lambda\in \Lambda})$ is clearly an  automorphism of $S$. 
However, we do not in general obtain all automorphisms of $S$ in this way. Similar problems therefore arise in regard to when $\aleph_0$-categoricity of $S$ passes to its induced group labeled bipartite graph (by which we mean the induced group labeled bipartite graph has an oligomorphic automorphism group).

An alternative next step could be to extend the scope of this section by considering  the $\aleph_0$-categoricity of Rees matrix semigroups over semigroups (or monoids), denoted $\mcal{M}^0[S;I,\Lambda;P]$, where again we assume $P$ is regular. 
Similarly we may define $\mcal{M}[S;I,\Lambda;P]$.  
 However, this task is as difficult as considering the $\aleph_0$-categoricity of all semigroups. 
 Indeed, if $S$ is a semigroup then $T=\mcal{M}^0[S^1;\{i\},\{\lambda\};(1)]$ is isomorphic to $S$ with both a zero and an identity adjoined, and by    \cite[Corollary 2.12]{Quinncat} $S$ is $\aleph_0$-categorical if and only if $T$ is $\aleph_0$-categorical.
 A second problem that arises is that the vital  Theorem \ref{iso c0s} only holds in the forwards direction for Rees matrix semigroups over semigroups. 
 As such we do not have an explicit description of the automorphism group of $\mcal{M}^0[S;I,\Lambda;P]$ via its components, and  many of the proofs of this section do not seem to be easily extendable. 
 In fact  the $\aleph_0$-categoricity of a Rees matrix semigroup over a semigroup $S$ does not necessarily pass to $S$, unlike for groups as shown in Proposition \ref{G and gamma cat}.
   For example, take any semigroup $S$  with zero element $\epsilon$, and consider $M=\mcal{M}[S;\{i\},\{\lambda\};(\epsilon)]$. Then $M$ is isomorphic to a null semigroup with zero element $(i,\epsilon,\lambda)$, which is $\aleph_0$-categorical by  \cite[Example 2.7]{Quinncat}; taking $S$ to be non $\aleph_0$-categorical gives our desired example. On the other hand, it can be easily shown that Proposition \ref{label cat rees} can be extended to Rees matrix semigroups over monoids. This allows us to build chains of $\aleph_0$-categorical semigroups as follows. Let $M$ be an $\aleph_0$-categorical monoid, and let $P$ be a $\Lambda\times I$ matrix over $\{0,1\}$ in which $\Gamma(P)$ is $\aleph_0$-categorical. Take $M_1=\mcal{M}^0[M;I,\Lambda;P]$, and inductively define $M_k=\mcal{M}^0[M_{k-1}^1;I,\Lambda;P]$ for $k>1$. Then each $M_k$ is $\aleph_0$-categorical, and $M_{k-1}$ embeds into $M_k$, for each $k\in \mathbb{N}$. 
   
\section*{acknowledgements}
The author  would   like to thank the referee for their valuable comments, which resulted in the vast improvement of Section 5. The suggestions to include greater detail about the completely simple case, and to give the interesting alternative proof of Lemma \ref{lemm:new} using Rees quotients were particularly useful.
Additionally the author would like to thank Prof. Victoria Gould for all her help at every stage of creating  this manuscript.

\end{document}